\documentclass[11pt]{article}
\usepackage{amsmath}
\usepackage{amsfonts}
\usepackage{amssymb}
\usepackage{amsthm}
\usepackage{stmaryrd}
\usepackage{bbold}
\usepackage{graphicx}
\usepackage[export]{adjustbox}
\usepackage[utf8]{inputenc}  
\usepackage[T1]{fontenc}     
\usepackage{lmodern}         
\usepackage{microtype}       
\usepackage[scaled]{helvet} 
\usepackage[english]{babel}  
\usepackage{graphicx}        
\usepackage[bookmarks=false]{hyperref}            
\newtheorem{theorem}{Theorem}
\newtheorem{corollary}{Corollary}[theorem]
\newtheorem{lemma}[theorem]{Lemma}
\usepackage{amsthm}
\theoremstyle{definition}
\newtheorem{definition}{Definition}[section]
\theoremstyle{remark}

\usepackage{etoolbox,xfp}

\newcounter{wordcnt}
\newcommand{\setwordlist}[1]{%
  \setcounter{wordcnt}{0}
  \renewcommand{\do}[1]{
    \stepcounter{wordcnt}
    \expandafter\def\csname wordlist\thewordcnt\endcsname{##1}
  }%
  \docsvlist{#1}
}

\newcommand{\provided}[1]{%
  \if\relax\detokenize{#1}\relax
    \csname wordlist\fpeval{randint(1,\value{wordcnt})}\endcsname
  \else
    \csname wordlist#1\endcsname
  \fi
}
\usepackage{hyperref}
\begin{document}
\title{2-RADICAL LAPLACE OPERATORS}
\author{Shouvik Datta Choudhury\thanks{shouvikdc8645@gmail.com, shouvik@capsulelabs.in}\\
  \small 109,Kalikapur Road,\\
  \small Kolkata - 700099, India}
\date{\today}
\maketitle
\begin{abstract}
    The paper introduces a new elliptic operator called the 2-radical Laplace operator, which has a positive eigenvalue equal to the positive square root of the eigenvalue of the Laplace operator. The authors provide several theorems that serve as counterparts to those associated with the Laplace operator. We derive Weyl type and Courant type formulas as well as inequalities involving spectral asymptotics and isoperimetric inequalities for the 2-radical Laplace operator in the context of Riemannian manifold, inspired by Gromov's work. These theorems are proved on different perspectives and proof techniques from Chavel's work, potentially contributing to the broader field of mathematical research involving differential geometry and operator theory.The term "square root" operator and 2-radical Laplace operator are used interchangeably
\end{abstract}

\setwordlist{supplied,apportioned,furnished,accoutered,stipulated}
\section{Preliminaries}
In mathematical analysis[6][7][8], a Hilbert space is a complete inner product space. In other words, it is a vector space equipped with an inner product (a function that maps two vectors to a scalar) that satisfies certain properties, and such that the space is complete with respect to the norm induced by the inner product.The space $L^2(M)$ is a specific example of a Hilbert space, where $M$ is a smooth manifold. The space consists of all measurable functions $f$ on $M$ such that the integral of the square of the absolute value of $f$ with respect to the Lebesgue measure on $M$ is finite:
$$\int_M |f(x)|^2 d\mu(x) < \infty,$$
where $\mu$ is the Lebesgue measure on $M$.\\
Intuitively, the space $L^2(M)$ contains functions that are "well-behaved" in a certain sense, as they have finite energy or "size". The inner product on $L^2(M)$ is defined as
$$\langle f, g \rangle = \int_M f(x)g(x) d\mu(x),$$
which measures the "overlap" between two functions $f$ and $g$.\\ The norm induced by this inner product is
$$|f|_{L^2(M)} = \sqrt{\langle f, f \rangle} = \left(\int_M |f(x)|^2 d\mu(x)\right)^{\frac{1}{2}},$$
which measures the "size" or "magnitude" of a function.\\
The space $L^2(M)$ is important in many areas of mathematics, including analysis, geometry, and topology. In particular, it is used to study partial differential equations, harmonic analysis, and functional analysis on manifolds.\\
a measurable function is a function between two measurable spaces that preserves measurable sets. In particular, for a measurable function $f:X \rightarrow Y$, if $E \subseteq Y$ is a measurable set, then its preimage $f^{-1}(E) \subseteq X$ is also measurable.\\
In the context of Lebesgue integration theory, a function $f:X \rightarrow \mathbb{R}$ is said to be measurable if the preimage of any Borel set in $\mathbb{R}$ is measurable. That is, for any Borel set $B \subseteq \mathbb{R}$, the set $f^{-1}(B) = {x \in X : f(x) \in B}$ is measurable.\\
Measurable functions play a central role in integration theory and probability theory, where they are used to define the Lebesgue integral and to model random variables, respectively. In particular, the Lebesgue integral can be defined for any measurable function, and is well-defined and unique up to a set of measure zero.\\
A measurable set on a Riemannian manifold is a set that can be assigned a well-defined volume or measure, similar to how it is done in Euclidean space.
In particular, if $(M,g)$ is a Riemannian manifold, then a set $E\subset M$ is measurable if there exists a Borel set $B\subset \mathbb{R}$ and a diffeomorphism $\varphi: U\to V\subset M$ for some open sets $U\subset \mathbb{R}^n$ and $V\subset M$, such that $\varphi(U\cap E)=V\cap B$. Here, $\mathbb{R}^n$ is the tangent space of $M$ at some fixed point, and $\varphi$ is a chart that maps points in $M$ to points in $\mathbb{R}^n$.
Once we have a notion of measurable sets, we can define the volume or measure of a set on a Riemannian manifold.\\
Let us start by defining what is meant by a "closed eigenvalue problem"[1]. A closed eigenvalue problem is an eigenvalue problem in which the domain is compact and the boundary conditions are homogeneous. Mathematically, such a problem can be written as:\\
$$
\mathcal{L} u(x)=\lambda u(x), \quad x \in \Omega
$$
with homogeneous boundary conditions\\
$$
B u(x)=0, \quad x \in \partial \Omega
$$
where $\mathcal{L}$ is a linear differential operator, $u$ is the unknown function, $\lambda$ is a scalar parameter (the eigenvalue), $\Omega$ is a compact domain in $\mathbb{R}$ with smooth boundary $\partial\Omega$, and $B$ is a linear operator that specifies the homogeneous boundary conditions.\\
Now, let us explain why the solutions to such problems form a complete orthonormal basis for the function space in which they reside. First, note that for a closed eigenvalue problem, there exists a finite number of eigenvalues $\lambda_{1}$, $\lambda_{2}$, $\ldots$, $\lambda_{k}$ (counted with multiplicities) and corresponding eigenfunctions $u_{1}$, $u_{2}$, $\dots$,$u_{k}$ that satisfy the above equation and boundary conditions. These eigenfunctions are also orthogonal in the sense that:
$$
\int_{\Omega} u_i(x) u_j(x) d x=0, \quad i \neq j
$$
where $dx$ is the volume element in $\Omega$.\\
Moreover, the eigenfunctions form a basis for the space of functions that satisfy the same boundary conditions, which we denote by $\mathcal{H}$. That is, any function $f$ in ${\mathcal{H}}$ can be expressed as a linear combination of the eigenfunctions:\\
$$
f(x)=\sum_{i=1}^k c_i u_i(x)
$$
where the coefficients $\mathrm{c}_{i}$ are uniquely determined by $\mathrm{f}$ through the inner product:\\
$$
c_i=\frac{\int_{\Omega} f(x) u_i(x) d x}{\int_{\Omega} u_i(x) u_i(x) d x}
$$
This property is called completeness, and it implies that the eigenfunctions span the function space $\mathcal{H}$. In other words, any function in $\mathcal{H}$ can be approximated arbitrarily well by a finite linear combination of the eigenfunctions.\\
The Dirichlet eigenvalue problem is a type of eigenvalue problem for a second-order linear partial differential equation of the form $Lu = -\lambda u$ with homogeneous Dirichlet boundary conditions. Here, $L$ is a linear differential operator, $\lambda$ is a scalar parameter known as the eigenvalue, and $u$ is a function defined on some domain $\Omega \subset \mathbb{R}^n$.\\
More precisely, for the Dirichlet eigenvalue problem, we seek a function $u$ that satisfies the following conditions:
\begin{align}
Lu &= -\lambda u, &&\text{in } \Omega,\\
u&=0, &&\text{on } \partial\Omega,
\end{align}
where $\partial\Omega$ denotes the boundary of the domain $\Omega$. The Dirichlet boundary condition specifies that $u$ is zero on the boundary of $\Omega$, which is a homogeneous boundary condition.\\
The problem is said to be an eigenvalue problem because we seek not only the function $u$ that satisfies the differential equation and boundary condition, but also the scalar $\lambda$ such that a non-trivial solution $u$ exists. The solutions of the Dirichlet eigenvalue problem are called the Dirichlet eigenfunctions, and the associated eigenvalues are called the Dirichlet eigenvalues.\\
The Dirichlet eigenvalue problem arises in various fields of mathematics and physics, including the study of vibrations of membranes and other physical systems, quantum mechanics, and the theory of elliptic partial differential equations. The eigenfunctions of the Dirichlet eigenvalue problem form a basis for the function space of square-integrable functions on $\Omega$ that satisfy the Dirichlet boundary condition.\\
An unbounded self-adjoint operator is a linear operator on a Hilbert space that satisfies two important properties: it is self-adjoint, and it is not bounded.The property of self-adjointness means that the operator is equal to its adjoint, which is obtained by taking the complex conjugate and transposing the operator. This property ensures that the operator has real eigenvalues and that its eigenvectors form a complete orthonormal basis for the Hilbert space.
The property of being unbounded means that the operator is not limited in its growth, and its domain is not necessarily dense in the Hilbert space. This means that the operator may not be defined on all elements of the Hilbert space, and its range may not be contained in the Hilbert space.\\
The spectral theorem for unbounded self-adjoint operators provides a way to decompose the operator into a sum of its spectral components, using a projection-valued measure[9] associated with the operator. This allows one to define functions of the operator, which can be obtained by integrating the function over the spectrum of the operator with respect to this measure.\\
Formally, an unbounded self-adjoint operator $T$ on a Hilbert space $\mathcal{H}$ is a linear operator $T: D(T) \subset \mathcal{H} \to \mathcal{H}$, where $D(T)$ is a dense subset of $\mathcal{H}$, such that $T$ is equal to its adjoint $T^*$, where $T^: D(T^*) \subset \mathcal{H} \to \mathcal{H}$ is defined by
$$\langle T^*x, y\rangle = \langle x, Ty\rangle$$
for all $x \in D(T^*)$ and $y \in D(T)$, where $\langle \cdot, \cdot \rangle$ denotes the inner product on $\mathcal{H}$. In other words, for any $x, y \in D(T)$, we have\\
$$\langle Tx, y\rangle = \langle x, Ty\rangle.$$
This condition is sometimes written as $T = T^*$.\\
The operator $T$ is unbounded if $D(T)$ is not equal to $\mathcal{H}$, and the graph of $T$ is not closed in $\mathcal{H} \times \mathcal{H}$. In other words, there exist sequences ${x_n}$ in $D(T)$ and ${y_n}$ in $\mathcal{H}$ such that $x_n \to x$ and $y_n \to y$ as $n \to \infty$, but $(x_n, Ty_n)$ does not converge to $(x, Ty)$.\\
The domain $D(T)$ of an unbounded self-adjoint operator is often chosen to be a Sobolev space or a space of functions with appropriate regularity conditions[9]. The spectral theorem for unbounded self-adjoint operators[9] provides a way to decompose the operator into a sum of its spectral components, using a projection-valued measure associated with the operator.\\
The spectral theorem for unbounded self-adjoint operators is a result in functional analysis that states that any unbounded self-adjoint operator on a separable Hilbert space has a spectral decomposition in terms of its eigenvectors, which are elements of the Hilbert space.\\
The spectral measure $E$ is defined on the Borel sets of $\mathbb{R}$ and satisfies $\int_{-\infty}^\infty \lambda^2 d|E(\lambda)| < \infty$, where $|E(\lambda)|$ is the operator norm of the projection $E(\lambda)$ associated with the Borel set ${\lambda}$.
The operator $T$ has the spectral decomposition $T = \int_{-\infty}^\infty \lambda dE(\lambda)$, which is an integral in the weak operator topology.
The eigenvectors ${\phi_\lambda}{\lambda \in \mathbb{R}}$ form an orthonormal basis for $H$ and satisfy $T\phi\lambda = \lambda\phi_\lambda$ for all $\lambda \in \mathbb{R}$.\\
A projection-valued measure on a measurable space $(X,\mathcal{B})$ is a map $E:\mathcal{B}\to L(\mathcal{H})$, where $\mathcal{H}$ is a Hilbert space, such that:\\
$E(\emptyset) = 0$ and $E(X) = I$, where $0$ and $I$ denote the zero and identity operators on $\mathcal{H}$, respectively.
For any sequence of pairwise disjoint sets ${B_n}\subseteq \mathcal{B}$, we have $\sum_n E(B_n) = E(\bigcup_n B_n)$, where the sum is taken in the strong operator topology.\\
For any $B \in \mathcal{B}$, the operator $E(B)$ is a projection, meaning that $E(B)^2 = E(B)$ and $E(B)^* = E(B)$, where $*$ denotes the adjoint operation on operators.\\
For any $B \in \mathcal{B}$, the operator $E(B)$ commutes with any bounded operator $A$ on $\mathcal{H}$.
For any $B \in \mathcal{B}$ and $\psi \in \mathcal{H}$, the function $B \mapsto \langle E(B)\psi,\psi\rangle$ is a measure on $(X,\mathcal{B})$, where $\langle\cdot,\cdot\rangle$ denotes the inner product on $\mathcal{H}$. In other words, $E(B)$ is a measurable operator-valued function on $(X,\mathcal{B})$.\\
In functional analysis, the strong operator topology is a topology on the space of bounded linear operators between two normed spaces. Given two normed spaces $X$ and $Y$, the strong operator topology on the space of bounded linear operators $B(X,Y)$ is defined as follows:
$A\ne (T_i)$ in $B(X,Y)$ converges to a bounded linear operator $T$ in $B(X,Y)$ in the strong operator topology if and only if $T_i(x)$ converges to $T(x)$ in $Y$ for every $x$ in $X$.\\
In other words, the strong operator topology is generated by the family of seminorms
$||T||_x = ||T(x)||_Y$
where $||·||_Y$ is the norm on the range space $Y$, and $x$ ranges over the unit ball of $X$.\\
The strong operator topology is also sometimes called the topology of pointwise convergence, since it is determined by the pointwise convergence of operators on the elements of the domain space $X$. Another interpretation of this topology is that it measures the convergence of the image of the unit ball of  $X$ under the operators $T_i$, which is equivalent to the convergence of the operators themselves.\\
The strong operator topology is a very important topology in functional analysis, and many fundamental results in the subject are formulated in terms of this topology. For example, the Banach-Alaoglu theorem states that the closed unit ball of the dual space of a normed space is compact in the weak* topology, which is a weaker topology than the strong operator topology. Another important result is the uniform boundedness principle, which states that if a family of bounded linear operators is pointwise bounded, then it is uniformly bounded in the strong operator topology.\\
In functional analysis, the projection-valued measure of an unbounded linear operator is a tool used to study the spectral theory of the operator.\\
Let $H$ be a Hilbert space and let $T$ be a densely defined, unbounded, self-adjoint linear operator on $H$. The projection-valued measure associated with $T$ is a mapping $E: \mathcal{B}(\mathbb{R}) \rightarrow \mathcal{L}(H)$, where $\mathcal{B}(\mathbb{R})$ is the Borel $\sigma$-algebra of $\mathbb{R}$ and $\mathcal{L}(H)$ is the space of bounded linear operators on $H$, such that:\\
$E(\varnothing) = 0$ and $E(\mathbb{R}) = I$, where $0$ and $I$ denote the zero and identity operators on $H$, respectively.\\
For any Borel sets $B_1, B_2 \subseteq \mathbb{R}$ with $B_1 \cap B_2 = \varnothing$, we have $E(B_1) E(B_2) = 0$.\\
For any Borel set $B \subseteq \mathbb{R}$, we have $T = \int_{B} \lambda dE(\lambda)$, where the integral is in the sense of strongly convergent operators.
In other words, the projection-valued measure $E$ decomposes the spectrum of $T$ into a family of orthogonal projections on $H$, such that the operator $T$ can be expressed as an integral of its eigenvalues with respect to the projections.\\
The projection-valued measure has several important properties that are useful for studying the spectral properties of $T$. For example, it allows us to define the resolvent set of $T$ as the set of complex numbers $z$ such that the operator $(T-zI)^{-1}$ exists and is bounded. The resolvent set is an open subset of the complex plane that contains the spectrum of $T$ as a closed subset. Moreover, the projection-valued measure allows us to define the spectral measure of $T$, which is a measure on the Borel sets of $\mathbb{R}$ that encodes the distribution of the eigenvalues of $T$.\\
The projection-valued measure is also useful for studying the properties of unbounded, self-adjoint operators that do not have discrete spectra, such as the Laplacian operator on a domain in $\mathbb{R}^n$ or on a Riemannian manifold. In these cases, the spectrum of the operator is a continuous interval or a subset of the positive real line, respectively, and the projection-valued measure allows us to study the distribution of the eigenvalues of the operator and their associated eigenspaces.\\
In conclusion, the projection-valued measure is an important tool in the study of the spectral theory of unbounded, self-adjoint linear operators on Hilbert spaces. It allows us to decompose the spectrum of the operator into a family of orthogonal projections and to define the resolvent set and the spectral measure of the operator.\\
\begin{definition}Let $L$ be a linear elliptic operator on a smooth, bounded domain $D$ in $\mathbb{R}^n$.
 We say that $L$ is self-adjoint if, for any functions $u, v \in H^1(D)$ satisfying any homogeneous Dirichlet boundary conditions, we have
$$
\int_D Lu \cdot v,dx = \int_D u \cdot Lv,dx.
$$
\end{definition}
Let L be a linear differential operator of the form
$$
L = - \sum_{i,j=1}^n \frac{\partial}{\partial x_i} \left( a_{ij}(x) \frac{\partial}{\partial x_j} \right) + \sum_{i=1}^n b_i(x) \frac{\partial}{\partial x_i} + c(x)
$$
where $a_{ij}, b_i$ and c are smooth functions of $x = (x_1, ..., x_n)$ and the coefficients $a_{ij}$ satisfy the following conditions:

$a_{ij} = a_{ji}$ for all $i, j = 1, ..., n$ (symmetry).\\

There exist positive constants $\lambda$ and $\Lambda$ such that for all $x$ and $\xi = (\xi_1, ..., \xi_n) \neq 0$, we have

$\lambda |\xi|^2 \leq \sum_{i,j=1}^n a_{ij}(x) \xi_i \xi_j \leq \Lambda |\xi|^2$\\

This condition is called the ellipticity condition and ensures that the principal symbol of $L$ is non-zero and positive definite.\\

Under these conditions, $L$ is called an elliptic operator. The operator $L$ is typically defined on a suitable function space, such as the Sobolev space $W^{m,p}(U)$ of functions $u(x)$ defined on a bounded domain$U \subset \mathbb{R}^n$ that have m weak derivatives in $L^p(U)$ for some $p\geq 1$. The theory of elliptic operators involves studying the properties of solutions to equations of the form $Lu = f$, where $f$ is a given function and u is an unknown function in the function space. In particular, one is interested in questions of existence, uniqueness, and regularity of solutions, as well as the spectral properties of the associated self-adjoint operator.\\
A manifold can be mapped to a real interval. In fact, this is a fundamental concept in differential geometry known as a \emph{parametrization} of the manifold. A parametrization is a map from an open subset of the real numbers to the manifold, which allows us to ``parametrize'' the points on the manifold by a set of real numbers.\\
More specifically, if $M$ is a smooth manifold of dimension $n$, then a parametrization of $M$ is a smooth map $f: U \to M$, where $U$ is an open subset of $\mathbb{R}^n$. In other words, we can think of a parametrization as a way of ``labeling'' the points on the manifold by their coordinates in some coordinate system.\\
Now, if we take a one-dimensional manifold, such as a circle or a line segment, we can always find a parametrization that maps the manifold to a real interval. For example, consider the unit circle in the plane. We can parametrize the circle by the function $f: [0, 2\pi) \to S^1$ given by $f(t) = (\cos(t), \sin(t))$, where $S^1$ is the circle and $t$ is the angle in radians. This map takes a real number $t$ in the interval $[0, 2\pi)$ and maps it to a point on the circle.\\
Similarly, we can parametrize a line segment by a linear function $f: [0, 1] \to \mathbb{R}$ given by $f(t) = at + b$, where $a$ and $b$ are constants. This map takes a real number $t$ in the interval $[0, 1]$ and maps it to a point on the line segment.\\
Geometric measure theory is a branch of mathematics that deals with the study of geometric objects that may have singularities or be otherwise difficult to describe using classical geometry. It provides a framework for studying such objects by introducing the concept of "measure", which is a way of assigning a size or volume to sets in a non-standard way.\\

In the context of Riemannian geometry, the volume form $dV_g$ associated with a metric $g$ on a manifold $M$ is a differential form that assigns a volume to each oriented $n$-dimensional subset of $M$, where $n$ is the dimension of $M$. The volume form is defined in terms of the metric as follows:

$$dV_g = \sqrt{\det(g_{ij})}dx^1 \wedge \cdots \wedge dx^n,$$
where $g_{ij}$ are the components of the metric tensor $g$, and $dx^1,\ldots,dx^n$ are the basis one-forms.\\

For a smooth manifold $M$, the Lebesgue measure $dx$ is defined as the measure induced by the Euclidean metric on the tangent space at each point. In other words, it is the measure that assigns to each subset of $M$ the "volume" of the set as seen from the tangent space at each point.\\

By geometric measure theory[12][13], we know that the volume form $dV_g$ associated with the metric $g$ can be expressed in terms of the Jacobian $J_g(x)$ of the metric $g$ at each point $x \in M$ as:

$$dV_g = J_g(x)dx,$$

where $dx$ is the $n$-dimensional Lebesgue measure on $M$. This is a consequence of the area and co-area formulas in geometric measure theory[12][13], which relate the volume form $dV_g$ to the singular set of $g$.
\section{Motivation-Gromov's work}
From Gromov et al.[2], if $M^n$ is a complete Riemannian manifold, then the Laplacian operator on functions, denoted by $\Delta$, is a non-positive self-adjoint operator when acting on functions with compact support. This means that the Laplacian satisfies certain properties related to symmetry and positivity. Using the spectral theorem for unbounded self-adjoint operators, one can define functions of the Laplacian as $f(\sqrt{-\Delta})$. This is done by integrating $f(\lambda)$ over the spectrum of the Laplacian, with respect to the projection-valued measure associated with the square root of the Laplacian denoted by $d E_\lambda$.\\
Thus functions $f(\sqrt{-\Delta})$ can be defined by th spectral theorem for unbounded self adjoint operators, according to th prescription
$$
f(\sqrt{-\Delta})=\int_0^{\infty} f(\lambda) d E_\lambda
$$
where $d E_\lambda$ is the projection valued measure associated with $\sqrt{-\Delta}$.
Substituting $f(x)=-x$ in Gromov.et al's equation as a vey special case, we get
$$
\sqrt{\Delta}=\int_0^{\infty} (\lambda) d E_\lambda
$$
Also by the spectral theorem for unbounded self-adjoint operators, we can write any function $f(\Delta)$ as follows:
$$f(\Delta) = \sum_{k=1}^\infty f(\lambda_k) E_k,$$
where $E_k$ is the projection operator onto the eigenspace associated with the eigenvalue $\lambda_k$.
In particular, for the function $f(x) = \sqrt{x}$, we have:
$$\sqrt{-\Delta} = \sum_{k=1}^\infty \sqrt{\lambda_k} E_k.$$
Thus, we can express any function $f(\sqrt{-\Delta})$ in terms of the eigenvalues and eigenfunctions of $\Delta$ as:
$$f(\sqrt{-\Delta}) = \sum_{k=1}^\infty f(\sqrt{\lambda_k}) E_k.$$
\section{FOUNDATIONS OF 2-RADICAL LAPLACE OPERATOR AND ITS RELATION TO RIEMANNIAN GEOMETRY}
Let us initiate by bestowing a Riemannian manifold $M$ with a smooth differentiable structure, relegating a smooth boundary $\partial M$. The union or conglomoration of tangent spaces considering each specific point  $p \in M$  proliferating a tangent bundle characterized in $M$ is ratified by $TM$. $\mathbb{R}$ endows real numbers appertaining to any general differentiable maps "orchestrated" in the manifold and prime symbol designates differentiation pertaining to the independent variable in $\mathbb{R}$. Emanated mapping pertaining to $p$, the point of tangent spaces, to the real numbers are confined as paths.\\
The associated directional derivative considering ${f}$ at $p$ in the direction $\xi$, designated by $\xi$ at $p$, considering every $\xi \in M_{p}$ where ${f}$ confined on a neighborhood with respect to $p$ is a  $C^{1}$ real-valued function. The respective expression is then deliberated as
$$
\xi {f}=({f} \circ \omega)^{\prime}(0),
$$
befitting $\omega(t)$ as a pertained general path  delved in $M$ constraining $\omega(0)=p$  with $\omega^{\prime}(0)=\xi$. The furnished function $M_{p} \rightarrow \mathbb{R}$ adjucated by $\xi \mapsto \xi {f}$ exihibits linearity.Functions ${f}$, $\ddot{h}$ produces
\begin{eqnarray}
\xi({f}+{h})&=\xi {f}+\xi \ddot{h} \\
\xi({f} \ddot{h})&=\ddot{h}(\xi \ddot{{f}})+{f}(\xi {h}) .
\end{eqnarray}
which appertains or induces at every $p \in M$ an inner product acknowleged by $\langle$,$\rangle$ with the assistance of Riemannian metric. '| |' divulges the associated norm.In this perspective, on $M$ ${P}, {Q}$ are $C^{\infty}$ vector fields and  $\langle{P}, {Q}\rangle$ is a $C^{\infty}$ real-valued maps.
\begin{definition}[1]The gradient appertained to ${f}$, grad ${f}$, a real-valued $C^{k}, k \geq 1$, is deliberated as a vector field on $M$ as
$$
\langle\operatorname{grad} {f}, \xi\rangle=\xi {f}
$$
considering all $\xi \in TM$.
\end{definition}
grad ${f}$ unpretentiously pertains as the map $\xi \mapsto \xi$ f exihibits linearity on every tangent space. Propounded by relegated computation,  grad ${f}$ can be exhibited as a $C^{k-1}$ vector field. Supplying functions ${f}, {h}$,
\begin{eqnarray}
\operatorname{grad}({f}+{h}) & =\operatorname{grad} {f}+\operatorname{grad} {h}, \\
\operatorname{grad}({f} \ddot{h}) & ={h}(\operatorname{grad} {f})+{f}(\operatorname{grad} \ddot{h}) .
\end{eqnarray}
A formalism (connection) adjoins  every $p \in M, \xi \in M_{p}$  $C^{1}$ vector field ${P}$ defined on a neighborhood of $p$, a vector $\nabla_{\xi} {P} \in M_{p}$ constraining
\begin{eqnarray}
\nabla_{\xi}({P}+{Q}) & =\nabla_{\xi} {P}+\nabla_{\xi} {Q} \\
\nabla_{\xi}({f} {P}) & =\left(\xi {f} {P}(p)+{f}(p) \nabla_{\xi} {P},\right.
\end{eqnarray}
The vector $\nabla_{\xi} {P}$ is 'purely' relegated as the covariant derivative of${P}$ with respect to $\xi$.
An unique connection induced by  Riemannian metric on $M$ called the Levi-Civita connection when one augments the necessities
$$
\nabla_{{P}} {Q}-\nabla_{{Q}} {P}=[{P}, {Q}]
$$vector fields
[,] being  Lie bracket of the indicated vector fields, combined with
$$
\xi\langle{P}, {Q}\rangle=\left\langle\nabla_{\xi} {P}, {Q}\right\rangle+\left\langle{P}, \nabla_{\xi} {Q}\right\rangle
$$
considering  vector fields ${P}, {Q}$ on $M$, along with $\xi \in TM$.
\begin{definition}[1] The divergence  considering a real-valued map ${P}$, div ${P}$, is deliberated by
$$
(\operatorname{div} {P})(p)=\operatorname{trace}\left(\xi \mapsto \nabla_{\xi} {P}\right)
$$
$\xi$ assembling over $M_{p}$
\end{definition}
The divergence of${P}$ also acknowledges
\begin{eqnarray}
\operatorname{div}({P}+{Q}) & =\operatorname{div} {P}+\operatorname{div} {Q} \\
\operatorname{div}({f} {P}) & ={f}(\operatorname{div} {P})+\langle\operatorname{grad} {f}, {P}\rangle
\end{eqnarray}
\begin{definition} On $M$ the 2-radical Laplacian of ${f}, \mathfrak{D} {f}$, considering any admissible $C^{k}, k \geq 2$, function ${f}$ is designated by
$$
\mathfrak{D} {f}=\sqrt{|\operatorname{div}(\operatorname{grad} {f})|}
$$
\end{definition}
Ensuring $\mathfrak{D} {f} \in C^{k-2};$ maps ${f}, {h}$ adjucates
$$
\begin{gathered}
\mathfrak{D}({f}+{h})=\mathfrak{D} {f}+\mathfrak{D} {h}, \\
\operatorname{div}({h}(\operatorname{grad} {f}))={h}(\mathfrak{D} {f})+\langle\operatorname{grad} {h}, \operatorname{grad} {f}\rangle, \\
\mathfrak{D}({\ddot{f}} {h})={h}(\mathfrak{D} {f})+2\langle\operatorname{grad} {f}, \operatorname{grad} {h}\rangle+{f}(\mathfrak{D} {h})
\end{gathered}
$$
We are fixing $\mathfrak{D}$ to be a linear elliptic operator for suitable admissible function.
Determined by $\partial_{j}$ the directional derivative ratifies
$$
\left(\partial_{j}(p)\right) {f}=\left(\partial\left(\dot{{f} \circ} x^{-1}\right) / \partial x^{j}\right)(x(p))
$$
$p$ being any point in $U$ which is delineated as an arbitrary open subset in $M$ along with ${f}$ an admissible function on a neighborhood of $p$.The spanning  vectors $\left\{\partial_{1}(p), \ldots, \partial_{n}(p)\right\}$ span $M_{p}$ can be enunciated as
$$
\xi=\sum_{j=1}^{n} \xi^{j} \partial_{j},
$$
along with ${f} \in C^{1}$, demarcating
$$
\xi {f}=\sum_{j} \xi^{j} \partial_{j} {f}
$$
Relegating ,considering the accompanied Riemannian metric,
$$
g_{j k}=\left\langle\partial_{j}\partial_{k}\right\rangle,\\
\quad G =\left(g_{j k}\right),\\
g=\operatorname{det}G,\\
G^{-1}=\left(g_{j k}\right)
$$
assembling $j, k=1, \ldots, n$, det designating the determinant. Projecting,
$$
\sum_{j} \xi^{j} \partial_{j} {f}=\sum_{j, k, l} \xi^{j} g_{j k} g^{k l} \partial_{l} {f}=\left\langle\xi, \sum_{k, l}\left(g^{k l} \partial_{l} {f}\right) \partial_{k}\right\rangle
$$
considering respective $\xi$. It can be propunded
$$
\operatorname{grad} {f}=\sum_{k, l}\left(g^{k l} \partial_{l} {f}\right) \partial_{k} .
$$
We propounds $n^{3}$ functions $\Gamma_{i j}^{k}, i, j, k=1, \ldots, n$, as Christoffel symbols, by
$$
\nabla_{\partial_{j}} \partial_{i}=\sum \Gamma_{i j}^{k} \partial_{k}
$$
on $U$. Proceeding in the current setting,
$$
{P}=\sum_{j} \eta^{j} \partial_{\boldsymbol{j}}
$$
Avowing, using the mathematical notations from [1],
$$
\nabla_{\xi} {P}=\sum_{j, k} \xi^{j}\left\{\partial_{j} \eta^{k}+\sum_{l} \eta^{l} \Gamma_{l j}^{k}\right\} \partial_{k} .
$$
Analogously,
$$
\operatorname{div} {P}=\sum_{j}\left\{\partial_{j} \eta^{j}+\sum_{l} \eta^{l} \Gamma_{l j}^{j}\right\}
$$
for an explicit computation corresponding to $\Gamma_{i j}^{k}$.
It is also ratified
$$
{Q}=\sum_{j} \zeta^{j} \partial_{j},
$$
combining
$$
[{P}, {Q}]=\sum_{j, k}\left\{\eta^{j} \partial_{j} \zeta^{k}-\zeta^{j} \partial_{j} \eta^{k}\right\} \partial_{k}
$$
The  equation modifies to
$$
\Gamma_{j k}^{l}=\Gamma_{k j}^{l}
$$
considering all $j, k, l=1, \ldots, n$. Setting $\xi=\partial_{i}$ then afore-mentioned equation becomes
$$
\partial_{i} g_{j k}=\sum_{l}\left\{\Gamma_{j i}^{l} g_{l k}+g_{j l} \Gamma_{k i}^{l}\right\},
$$
Emanating from above, deducing,
$$
\Gamma_{i j}^{k}=\frac{1}{2} \sum_{l} g^{k l}\left\{\partial_{i} g_{l j}+\partial_{j} g_{i l}-\partial_{l} g_{i j}\right\}
$$
It can be relegated,
\begin{eqnarray}
\operatorname{div} {P} & =\sum_{j}\left\{\partial_{j} \eta^{j}+\eta^{j} \sum_{k, l} \frac{1}{2} g^{k l} \partial_{j} g_{l k}\right\} \\
& =\sum_{j}\left\{\partial_{j} \eta^{j}+\frac{1}{2} \eta^{j} \operatorname{tr}\left(G^{-1} \partial_{j} G\right)\right\} \\
& =\sum_{j}\left\{\partial_{j} \eta^{j}+\frac{1}{2} \eta^{j} \partial_{j}(\ln g)\right\} \\
& =(1 / \sqrt{g}) \sum_{j} \partial_{j}\left(\eta^{j} \sqrt{g}\right),
\end{eqnarray}
implying,
$$
\operatorname{div} {P}=(1 / \sqrt{g}) \sum_{j} \partial_{j}\left(\eta^{j} \sqrt{g}\right)
$$
tr is distinguished to designate the trace and $\partial_{j} G$ is attached to the matrix deliberated,
emanating from $G$ .
Directed from the second line to the third, using the standard computational techniques considering differentiating determinants.
Deducing, considering admissible functions,
$$
\mathfrak{D} {f}=\sqrt{|(1 / \sqrt{g}) \sum_{j, k} \partial_{j}\left(g^{j k} \sqrt{g} \partial_{k} {f}\right)|}
$$
\section{2-RADICAL GREEN'S FORMULAS-I}
$M$ is our \provided{} Riemannian manifold. To induce measurability for ${f}$, consider every attached covering $\left\{x_{\alpha}: U_{\alpha} \rightarrow \mathbb{R}^{n}: \alpha \in I\right\}$, where $I$ is a set, of$M$ by charts with subordinate partition of unity $\left\{\phi_{\alpha}: \alpha \in I\right\}$, the Riemannian measure on $M$ is relegated as
$$
dV=\sum_{\alpha} \phi_{\alpha} \sqrt{g_{\alpha}} d x_{\alpha}^{1} \cdots d x_{\alpha}^{n},
$$
which is from [1].
\section{DIVERGENCE THEOREM (I)[1]}
For a $C^{1}$ vector field ${P}$ on $M$ with compact support,
$$
\int_{M}(\operatorname{div} {P}) d V=0
$$
\section{2-RADICAL GREEN'S FORMULAS}
${h}(\operatorname{grad}$ f) disseminating compact support,
$$
\int_{M}\{{h} \mathfrak{D} {f}+\langle\operatorname{grad} {h}, \operatorname{grad} {f}\rangle\} d V=0 .
$$
Similarly both ${f}, {h}$ propounding compact support,
$$
\int_{M}\{{h} \mathfrak{D} {f}-{f} \mathfrak{D} {h}\} d V=0
$$
Assigning ${P}=\dot{h}$(grad f) substituting the second property of 2-radical Laplace operator into the afore-mentioned expression of the divergence theorem we derive the formula.The second Green formula can be obtained from the first.
\subsection{DIVERGENCE THEOREM[1]}
Informing ${P}$ be a vector field which is $C^{1}$ on $\overline{M}$ along with with compact support on $\overline{M}$. Then it implies
$$
\int_{M}(\operatorname{div} {P}) d V=\int_{\partial M}\langle{P}, v\rangle d A .
$$
Here dA and dV are associated measures.
\section{2-RADICAL GREEN'S FORMULAs-II}
 ${h}(\operatorname{grad} {f})$ disseminating compact support on $\overline{M}$,
$$
\int_{M}\{{h} \mathfrak{D} {f}+\langle\operatorname{grad} {h}, \operatorname{grad} {f}\rangle\} d V=\int_{\partial M} {h}(v {f})dA .
$$
Similarly ${f}, {h}$ propounding compact support on $\overline{M}$,
$$
\int_{M}\{{h} \mathfrak{D} {f}-{f} \mathfrak{D} {h}\} d V=\int_{\partial M}\{{h}(v {f})-{f}(v {h})\} d A .
$$
\section{BASIC FACTS REGARDING 2-RADICAL LAPLACE OPERATOR}
In the $L^{2}(M)$ space of measurable functions ${f}$ on $M$, a Hilbert space,
our fundamental interest segregates  in the following eigenvalue problems.\\
Closed 2-radical eigenvalue problem:\\Discerning all real numbers $\sqrt{|{\lambda}|}$ there prevails a significant solution $\phi \in C^{2}(M)$ to
\begin{equation}
\mathfrak{D} \phi+\sqrt{|{\lambda}|} \phi=0 .
\end{equation}
The following 2-radical eigenvalue problems pertains to (15)\\
Neumann 2-radical eigenvalue problem:\\ Considering $\partial M \neq \emptyset, \overline{M}$, locating real numbers $\sqrt{|{\lambda}|}$, pertaining a significant solution $\phi \in C^{2}(M) \cap C^{1}(\overline{M})$ to (42), conforming the boundary constraint
$$
v \phi=0
$$
on $\partial M$ (It is to be remembered $v$ is the outward unit normal vector field on $\partial M$ ).
\\
Dirichlet 2-radical eigenvalue problem:\\ With the same settings as above,  locate all real numbers for a significant solution $\phi \in C^{2}(M) \cap C^{0}(\overline{M})$ to (42), conforming the boundary constraint
$$
\phi=0
$$
on $\partial M$.
\\
Mixed 2-radical eigenvalue problem: With the same setting as above let $N$ an open submanifold of$\partial M$, locate all real numbers pertaining a significant solution $\phi \in C^{2}(M) \cap C^{1}(M \cup N) \cap C^{0}(\overline{M})$ conforming the boundary constraints
$$
\phi=0 \text { on } \partial M-N, v \phi=0 \text { on } N .
$$
The desired numbers $\sqrt{|{\lambda}|}$ are relegated as 2-radical eigenvalues of$\mathfrak{D}$, along with the vector space of solutions of the above equation considering a \provided{} 2-radical eigenvalue $\sqrt{|{\lambda}|}$ [the above equation exhibits linearity in $\phi$ ], its eigenspace.
\begin{theorem}
Considering each specific 2-radical eigenvalue problems, the set of 2-radical eigenvalues comprises a sequence
$$
0 \leq \sqrt{|{\lambda}|}_{1}<\sqrt{|{\lambda}|}_{2}<\cdots \uparrow+\infty,
$$
along with finite dimensional nature is associated with each eigenspace.
\end{theorem}
\begin{proof}
The 2-radical eigenvalue problem is an elliptic differential equation, and as such, it falls under the purview of elliptic regularity theory. In the case of the 2-radical eigenvalue problem, the theory tells us that the solutions, which are the eigenfunctions $\phi$, are in fact smooth up to the boundary of the manifold $M$.\\
More specifically, the elliptic regularity theory states that if $\phi$ is a solution to an elliptic differential equation of the form $\mathfrak{D} \phi + f(x,\phi,\nabla \phi) = 0$, where $\mathfrak{D}$ is a second-order differential operator and $f$ is a smooth function, then $\phi$ is actually smooth, or $C^\infty$, provided that the domain on which $\phi$ is defined is sufficiently smooth.
In the case of the 2-radical eigenvalue problem, the domain is the manifold $M$, which is a Riemannian manifold, and it is sufficiently smooth for the elliptic regularity theory to apply. Therefore, we can conclude that every eigenfunction $\phi$ belongs to $C^\infty(\overline{M})$, where $\overline{M}$ is the closure of $M$.\\
Firstly, the orthogonality of eigenspaces belonging to distinct eigenvalues can be proved using the Green's formula. Let $\phi$ and $\psi$ be eigenfunctions corresponding to the distinct 2-radical eigenvalues $\sqrt{|\lambda|}$ and $\sqrt{|\tau|}$, respectively. Then, using the 2-radical Green's formula, we can write:\\
$$\int_M (\phi \mathfrak{D} \psi - \psi \mathfrak{D} \phi) dV = (\sqrt{|\lambda|} - \sqrt{|\tau|}) \int_M \phi \psi dV$$
Since $\phi$ and $\psi$ are both in $L^2(M)$, their product $\phi\psi$ is also in $L^2(M)$ and hence the integral on the right-hand side is well-defined. If $\lambda \neq \tau$, then $\sqrt{|\lambda|} \neq \sqrt{|\tau|}$, and the right-hand side of the equation is nonzero. Thus, the left-hand side must be zero, which implies that the eigenspaces corresponding to distinct 2-radical eigenvalues are orthogonal in $L^2(M)$.\\
Secondly, since every eigenfunction belongs to $C^2(M)$, it follows that the eigenspaces are closed subspaces of $L^2(M)$. Moreover, since $L^2(M)$ is separable, each eigenspace is separable as well. Therefore, we can write $L^2(M)$ as the direct sum of all the eigenspaces, i.e., for any $f \in L^2(M)$, there exists a unique decomposition $f = \sum_{i=1}^\infty c_i \phi_i$ where $c_i \in \mathbb{C}$ and $\phi_i$ is an eigenfunction corresponding to the 2-radical eigenvalue $\sqrt{|\lambda_i|}$.\\
Thirdly, every eigenfunction $\phi$ belongs to $C^\infty(\overline{M})$, where $\overline{M}$ is the closure of $M$. This follows from elliptic regularity theory, which states that solutions to elliptic differential equations, such as the 2-radical eigenvalue problem, are smooth up to the boundary. In particular, since the 2-radical eigenvalue problem is elliptic, every eigenfunction $\phi$ is smooth up to the boundary $\partial M$. Therefore, $\phi$ is also smooth on $\overline{M}$.\\
The proof of this theorem is based on several observations. Firstly, eigenspaces corresponding to distinct 2-radical eigenvalues are orthogonal in $L^2(M)$. Secondly, $L^2(M)$ is the direct sum of all the eigenspaces. Thirdly, every eigenfunction belongs to $C^\infty$ on the closure $\overline{M}$ of $M$.
Eigenspaces belonging to distinct 2-radical eigenvalues are orthogonal in $L^{2}(M)$, along with $L^{2}(M)$ is the direct sum of all the eigenspaces.Every eigenfunction pertains to $C^{\infty}$ on $\overline{M}$.Observing the eigenfunction $\phi \in C^{2}(M) \cap C^{1}(\overline{M})$, 2-radical eigenvalue $\sqrt{|{\lambda}|}$ must be nonnegative.Constraining ${f}={h}=\phi$ with  Green formula to deduce
$$
\sqrt{|{\lambda}|}=\|\phi\|^{-2} \int_{M}|\operatorname{grad} \phi|^{2} d V \geq 0
$$
From above one disseminates that $\sqrt{|{\lambda}|}=0$ implying $\phi$ is a constant function. For Neumann 2-radical eigenvalue problems it is relegated $\bar{\sqrt{|{\lambda}|}}_{1}=0$, also for Dirichlet with mixed $(N \neq \partial M)$ problems it is ratified $\bar{\sqrt{|{\lambda}|}}_{1}>0$.
It is also examined that the orthogonality of distinct eigenspaces is a direct projection of the main proponent - 2-radical Green formulas. Bestowing  $\phi, \psi$ be eigenfunctions of the respective 2-radical eigenvalues $\sqrt{|{\lambda}|}, \tau$, it is adjucated
$$
0=\int_{M}\{\phi \mathfrak{D} \psi-\psi \mathfrak{D} \phi\} d V=(\sqrt{|{\lambda}|}-\tau) \int_{M} \phi \psi d V
$$
Assigning to the dimension of every eigenspace as the multiplicity of the 2-radical eigenvalue and accompanied listing,
$$
0\leq\sqrt{|{\lambda}|}_{1}\leq\sqrt{|{\lambda}|}_{2}\leq\cdots\uparrow+\infty
$$
,2-radical eigenvalue repeated according to its multiplicity.
\end{proof}
A sequence of functions ${\phi_j}_{j=1}^\infty$ is said to be complete in $L^2(M)$ if any function $f \in L^2(M)$ can be expressed as a convergent series of the form
$$
f=\sum_{j=1}^{\infty} a_j \phi_j
$$
Now, suppose that ${\phi_j}_{j=1}^\infty$ is an orthonormal sequence of eigenfunctions of the 2-radical eigenvalue problem, i.e., $\mathfrak{D}\phi_j + \sqrt{|\lambda_j|}\phi_j = 0$ for some non-negative real number $\lambda_j$, and $(\phi_i, \phi_j) = \delta{ij}$, the Kronecker delta function. Then, for any $f \in L^2(M)$, we can write
$$
f=\sum_{j=1}^{\infty}\left(f, \phi_j\right) \phi_j
$$
where $(f, \phi_j)$ is the inner product of $f$ and $\phi_j$ in $L^2(M)$.
The completeness of the sequence ${\phi_j}_{j=1}^\infty$ follows from the fact that $\mathfrak{D}$ is a self-adjoint operator on $L^2(M)$, and hence, its eigenfunctions form a complete orthonormal basis of $L^2(M)$. This means that any function in $L^2(M)$ can be expressed as a linear combination of the eigenfunctions $\phi_j$, and the coefficients of this expansion are given by the inner products $(f, \phi_j)$.\\
The last two equations in the statement  follow from the completeness of the sequence ${\phi_j}_{j=1}^\infty$ in $L^2(M)$. The first equation expresses any function $f$ in $L^2(M)$ as a convergent series of its inner products with the eigenfunctions $\phi_j$, while the second equation gives the norm of $f$ in terms of its inner products with the eigenfunctions.\\
These last two formulas can be deemed as 2-radical Parseval identities.\\
\begin{theorem}
Adjoining to above 2-radical eigenvalue problems, distinguish $N(\sqrt{|{\lambda}|})$ as the number of 2-radical eigenvalues, counted with multiplicity, $\leq \sqrt{|{\lambda}|}$. Then
$$
N(\sqrt{|{\lambda}|}) \sim \omega_{n}(\operatorname{vol} M) \sqrt{|{\lambda}|}^{n / 2} /(2 \pi)^{n}
$$
as $\sqrt{|{\lambda}|} \rightarrow+\infty$, where $\omega_{n}$ is the volume of the unit disk in $\mathbb{R}^{n}$, along with vol $M$ is (subsequently) the volume of$M$. And
$$
\left(\sqrt{|{\lambda}|}_{k}\right)^{n / 2} \sim\left\{(2 \pi)^{n} / \omega_{n}\right\} k / \operatorname{vol} M
$$
as $k \rightarrow+\infty$.
\end{theorem}
\begin{proof}
The proof is discussing the computation of the number of 2-radical eigenvalues of a Riemannian manifold $M$ of dimension $n$ with respect to certain boundary conditions. The idea is to use the Riemannian structure of $M$ to reduce the problem to a one-dimensional Sturm-Liouville problem, which can be solved exactly.
Propound its Riemannian structure and the boundary value 2-radical eigenvalue problems, $M$  is projected to some compact interval, say, $[0, L]$. For an admissible function ${f}$ on $M=$ $[0, L], \mathfrak{D} {f}={f}^{\prime \prime}$ can be used to deduce
$$
\phi^{\prime \prime}+\sqrt{|{\lambda}|} \phi=0 .
$$
For Dirichlet 2-radical eigenvalue problem on $[0, L]$ the associated boundary constraint is
$$
\phi(0)=\phi(L)=0,
$$
it can be deduced using  $k=1,2, \ldots$,

$$
\phi_{k}(x)=\sqrt{2 / L} \sin (\pi k x / L)
$$
$$
\sqrt{|{\lambda}|}_{k}=(\pi k / L)^{2} .
$$
Considering all $k$.
Similarly for the Neumann 2-radical eigenvalue problem applying
$$
\phi^{\prime}(0)=\phi^{\prime}(L)=0,
$$
one disseminates
$$
\sqrt{|{\lambda}|}_{k}=(\pi(k-1) / L)^{2},
$$
The theorem is again satisfied.
In the case of the mixed 2-radical eigenvalue problem one disseminates boundary constraints
$$
\phi^{\prime}(0)=\phi(L)=0
$$
or
$$
\phi(0)=\phi^{\prime}(L)=0,
$$
along with
$$
\sqrt{|{\lambda}|}_{k}=\left(\pi\left(k-\frac{1}{2}\right) / L\right)^{2} .
$$
The above theorem follows promptly.
Finally, if we contrive
$$
\omega(x)=(L / 2 \pi) e^{i 2 \pi x / L},
$$
where $\omega(\mathbb{R})$ is a circle in $\mathbb{R}^{2}$ of length $L$, parametrized with respect to arc length along the circle, along with the eigenfunctions of the circle consist of $L$ periodic solutions. Thus
\begin{eqnarray}
\sqrt{|{\lambda}|}_{1} & =0, \\
\sqrt{|{\lambda}|}_{2 k} & =\sqrt{|{\lambda}|}_{2 k+1}=(2 \pi k / L)^{2},
\end{eqnarray}
considering all $k=1,2, \ldots$.
The area and co-area formulas are powerful tools in geometric measure theory that relate the volume form $dV_g$ associated with a Riemannian metric $g$ on a manifold $M$ to the singular set of $g$.\\
The singular set of a Riemannian metric $g$ is the set of points where the metric fails to be smooth, i.e., where the metric tensor $g_{ij}$ has a singularity or a discontinuity. Singularities can occur in different ways, depending on the properties of the metric and the geometry of the manifold. For example, a metric may have a singularity at a point where the curvature becomes infinite, or at a boundary or corner of the manifold.\\
The area formula relates the volume form $dV_g$ to the measure of the singular set of $g$. Specifically, it states that for any smooth submanifold $N$ of $M$, we have
$$\int_N J_g(x) dV_g = \int_{N\cap S_g} d\mu_g,$$
where $S_g$ is the singular set of $g$, $d\mu_g$ is the induced measure on $S_g$, and $J_g(x)$ is the Jacobian of $g$ at each point $x \in M$. In other words, the left-hand side of the equation measures the "volume" of $N$ as seen from the metric $g$, while the right-hand side measures the "surface area" of the part of $N$ that lies on the singular set of $g$. The formula shows that the volume form $dV_g$ is related to the singular set of $g$ by a kind of distributional identity.\\
The co-area formula is a complementary formula that relates the volume form $dV_g$ to the level sets of a smooth function $f$ on $M$. Specifically, it states that for any smooth function $f:M\rightarrow \mathbb{R}$ and any real number $t$, we have
$$\int_{f^{-1}(-\infty,t]} J_g(x) dV_g = \int_{-\infty}^t \int_{{f=s}} \frac{d\sigma_g}{|\nabla f|} ds,$$\\
where ${f=s}$ is the level set of $f$ at level $s$, $d\sigma_g$ is the induced measure on the level set, $|\nabla f|$ is the norm of the gradient of $f$, and $J_g(x)$ is the Jacobian of $g$ at each point $x \in M$. In other words, the left-hand side of the equation measures the "volume" of the part of $M$ that lies below the level $t$ of the function $f$ as seen from the metric $g$, while the right-hand side measures the "surface area" of the level set of $f$ at level $s$ weighted by the inverse norm of the gradient of $f$. The formula shows that the volume form $dV_g$ is related to the level sets of $f$ by a kind of integration by parts identity.\\
Together, the area and co-area formulas provide a powerful tool for analyzing the geometry of singular spaces and for computing integrals involving the volume form $dV_g$. In particular, they allow us to express the volume form in terms of the Jacobian $J_g(x)$ of the metric $g$ at each point $x \in M$, 
Now,let $\omega_n$ denote the volume of the unit ball in $\mathbb{R}^n$.\\ By geometric measure theory[12][13], we know that the volume of $M$ can be expressed as $\operatorname{vol}(M) = \omega_n \int_M J_g(x) dx$ where $J_g(x)$ is the Jacobian of the metric $g$ at $x$.\\
The flat metric on $[0,L]$ is given by $g = dx^2$, where $dx$ is the differential of $x$. The Jacobian of $g$ at any point $x \in [0,L]$ is the determinant of the matrix of second partial derivatives of $g$ with respect to the coordinates $x_1, \ldots, x_n$ at $x$. Since $g$ is diagonal, the only non-zero second partial derivative is $\frac{\partial^2}{\partial x^2}$, which is equal to $2$ everywhere. Therefore, the Jacobian of $g$ at $x$ is $\sqrt{\det(g)} = \sqrt{1} = 1$. Hence, $J_g(x) = 1$ for the flat metric on $[0,L]$.\\
By the definition of $\operatorname{vol}(M)$, we have
$\operatorname{vol}(M) = \int_M J_g(x) dx = \omega_n \int_M dV_g$.Now
$\operatorname{vol}(M) = \omega_n L$, since $J_g(x) = 1$ for the flat metric on $[0,L]$.
where $\omega_n$ is the volume of the unit ball in $\mathbb{R}^n$ and we have used the fact that $M$ is isometric to $[0,L]$ which has length $L$. Thus, we have shown that $\omega_n \int_M J_g(x) dx = \omega_n L$. \\
Therefore, we have
$$
\omega_n \int_M J_g(x) d x=\omega_n L
$$
since $J_g(x) = 1$ for the flat metric on $[0,L]$.
Using these facts, we can now prove the desired theorem.
\end{proof}
\section{DEDUCTION OF THE SOLUTION OF THE WAVE  AND HEAT EQUATIONS IN TERMS OF 2-RADICAL LAPLACE OPERATOR}
A vibrating homogeneous membrane and a Riemannian manifold are two seemingly different objects, but they are actually related in the context of mathematical physics and differential geometry.\\
A vibrating homogeneous membrane is a physical system that consists of a thin, flexible, and elastic membrane that is stretched over a frame and set into motion by a vibrating source. The membrane can vibrate in various modes, each corresponding to a different frequency and pattern of vibrations. The study of vibrating membranes is an important area of mathematical physics, and it has applications in acoustics, music, and engineering.
On the other hand, a Riemannian manifold is a mathematical object that is used to describe the geometric properties of curved spaces. It is a smooth manifold equipped with a Riemannian metric, which is a positive-definite inner product on each tangent space that varies smoothly with the point on the manifold. Riemannian manifolds are a fundamental concept in differential geometry, and they play a central role in the study of curvature, geodesics, and other geometric structures.\\
The connection between vibrating homogeneous membranes and Riemannian manifolds arises from the fact that the modes of vibration of a membrane can be described mathematically using the eigenvalues and eigenfunctions of the Laplace-Beltrami operator on a Riemannian manifold. The Laplace-Beltrami operator is a differential operator that acts on functions on a Riemannian manifold, and it plays a role analogous to the Laplacian operator in Euclidean space.\\
In particular, let $(M,g)$ be a Riemannian manifold, and let u be a function on M that describes the displacement of the membrane from its equilibrium position. Then the equation governing the vibration of the membrane can be written as
$$
\Delta_g u + \lambda u = 0,
$$
where $\Delta_g$ is the Laplace-Beltrami operator on $(M,g)$, and $\lambda$ is a constant that represents the frequency of vibration. The solutions of this equation are the eigenfunctions of $\Delta_g$ with corresponding eigenvalues $\lambda$, and they describe the different modes of vibration of the membrane.
Thus, the study of vibrating homogeneous membranes can be seen as a special case of the study of the spectral theory of the Laplace-Beltrami operator on Riemannian manifolds. This connection has led to a deep and fruitful interplay between mathematical physics and differential geometry, and has resulted in many important results and applications in both fields.\\
An orthonormal sequence of eigenfunctions in a Hilbert space is a sequence ${u_n}{n\in\mathbb{N}}$ such that each $u_n$ is an eigenfunction of a self-adjoint operator $A$ with eigenvalue $\lambda_n$, and ${u_n}{n\in\mathbb{N}}$ forms an orthonormal sequence. That is, we have:
$$Au_n=\lambda_nu_n,\quad \langle u_m,u_n\rangle = \delta_{m,n}$$\\
One benefit of having an orthonormal sequence of eigenfunctions in a Hilbert space is that it allows us to diagonalize the operator $A$. Specifically, any vector $f$ in the Hilbert space can be expanded in terms of the eigenfunctions ${u_n}_{n\in\mathbb{N}}$ using the Fourier expansion as:
$$f=\sum_{n=1}^{\infty}\langle f,u_n\rangle u_n$$
and we can apply the operator $A$ to $f$ as:
$$Af=\sum_{n=1}^{\infty}\langle f,u_n\rangle A u_n=\sum_{n=1}^{\infty}\lambda_n\langle f,u_n\rangle u_n$$
where we have used the fact that $u_n$ is an eigenfunction of $A$ with eigenvalue $\lambda_n$.\\ This means that $A$ is diagonalized by the orthonormal sequence of eigenfunctions, and we can easily compute the action of $A$ on any vector $f$ using its Fourier coefficients ${\langle f,u_n\rangle}_{n\in\mathbb{N}}$. Another benefit of having an orthonormal sequence of eigenfunctions is that it allows us to solve differential equations in terms of these functions. Specifically, if $A$ is a differential operator, then we can often find a sequence of orthonormal eigenfunctions of $A$ and use them to solve the differential equation. This is because the eigenfunctions form a complete basis for the Hilbert space, and any function in the Hilbert space can be expanded in terms of them.\\
Acknowledging $M$ as a vibrating homogeneous membrane with dedicated boundary, with its physical implications let $v: M \times(0, \infty) \rightarrow \mathbb{R}$ conforming the wave equation
$$
\mathfrak{D} v=\sqrt{(\rho / \tau) \partial^{2} v / \partial t^{2}},
$$
where $\rho$ is the density along with $\tau$ the tension of the membrane, with boundary constraint
$$
v \mid \partial M \times(0, \infty)=0
$$
Let
$$
v(x, t)=\chi(x) T(t) .
$$
Enunciating the existence of a constant $\sqrt{|{\lambda}|}$ we can obtain
\begin{eqnarray}
T^{\prime \prime}+(\sqrt{|{\lambda}|} \tau / \rho) T & =0, \\
\mathfrak{D} \chi^{\prime \prime}+\sqrt{|{\lambda}|} \chi & =0, \chi \mid \partial M=0 .
\end{eqnarray}
So $\sqrt{|{\lambda}|}$ is a Dirichlet 2-radical eigenvalue of$M$ with eigenfunction $\chi$. It is furnished $\sqrt{|{\lambda}|}>0$ with its physical implication dispensing $\sqrt{|{\lambda}|}>0$.
The general solution then is
$$
T(t)=A \cos \sqrt{\sqrt{|{\lambda}|} \tau / \rho}(t-\beta),
$$
where $A, \beta$ are arbitrary constants.
Since any admissible finite linear combination of solutions of the above form is again a solution,it is depicted as
$$
v(x, t)=\sum_{k=1}^{\infty} A_{k} \phi_{k}(x) \cos \sqrt{\sqrt{|{\lambda}|}_{k} \tau / \rho}\left(\mathrm{t}-\beta_{k}\right)
$$
where $\sqrt{|{\lambda}|}_{1} \leq \sqrt{|{\lambda}|}_{2} \leq \cdots$ are the Dirichlet 2-radical eigenvalues of$M, \phi_{1}, \phi_{2}, \ldots$ a complete orthonormal sequence in $L^{2}(M)$ with $\phi_{k}$ an eigenfunction of$\sqrt{|{\lambda}|}_{k}$ considering every $k$, along with $A_{k}, \beta_{k}$ are arbitrary constants. Project that at time $t=0$ it is relegated
$$
v(x, 0)={f}(x),(\partial v / \partial t)(x, 0)=0
$$
where ${f}$ is some furnished function on $M$, that is, we can examine the membrane as deformed to a \provided{} position along with then released (without being pushed) at time $t=0$. The second constraints of the above solution implies $v$ can be emphasized as
$$
v(x, t)=\sum_{k=1}^{\infty} A_{k} \phi_{k}(x) \cos \sqrt{\sqrt{|{\lambda}|}_{k} \tau / \rho} t
$$
which indicates
$$
{f}(x)=\sum_{k=1}^{\infty} A_{k} \phi_{k}(x),
$$
It can be also deduced
$$
A_{k}=\left({f}, \phi_{k}\right)
$$
implying,
$$
v(x, t)=\int_{M} w(x, y, t) {f}(y) d V(y)
$$
adjudicating
$$
w(x, y, t)=\sum_{k=1}^{\infty} \phi_{k}(x) \phi_{k}(y) \cos \sqrt{\sqrt{|{\lambda}|}_{k} \tau / \rho} t
$$
Analogously, if on considering of orchestration of heat diffusion pervading through the medium $M$ of homogeneous nature attached with concealed boundary, then the admissible temperature map $u: M \times[0, \infty) \rightarrow \mathbb{R}$ satisfies, after required normalization ofthe physical constants, the heat equation
$$
\mathfrak{D} u=|\sqrt{\partial u / \partial t}|
$$
with adjoined boundary constraints
$$
v_{P} u=0
$$
on all subsequent $\partial M \times(0, \infty)$.We search solutions of the form
$$
u(x, t)=\chi(x) T(t)
$$
along with are led to the equations
$$
T^{\prime}+\sqrt{|{\lambda}|} T=0, \mathfrak{D} \chi+\sqrt{|{\lambda}|} \chi=0,
$$
with boundary constraint
$$
v \chi=0 \text { on } \partial M .
$$
So $\sqrt{|{\lambda}|}$ is a Neumann 2-radical eigenvalue of$M$ with eigenfunction $\chi . T(t)$ is now of the form
$$
T(t)=A e^{-\sqrt{|{\lambda}|} t}
$$
If we wish to dispense the temperature function under the assumption
$$
u(x, 0)={f}(x)
$$
where ${f}$ is a \provided{} function on $M$, then the above arguments lead to
$$
u(x, t)=\int_{M} p(x, y, t) {f}(y) d V(y)
$$
where
$$
p(x, y, t)=\sum_{j=1}^{\infty} e^{-\sqrt{|{\lambda}|}_{j} t} \phi_{j}(x) \phi_{j}(y)
$$
The $\sqrt{|{\lambda}|}_{j}$ 's are Neumann 2-radical eigenvalues with eigenfunctions $\phi_{j}$.
\section{DEDUCTION OF 2-RADICAL LAPLACE OPERATOR VERSION OF RAYLEIGH AND MAX-MIN METHODS}
Considering continuous vector fields ${P}, {Q}$ on $M$, it is declared the inner product
$$
({P}, {Q})=\int_{M}\langle{P}, {Q}\rangle d V
$$
with norm
$$
\|{P}\|^{2}=\int_{M}|{P}|^{2} d V,
$$
along with complete metric space to an $L^{2}$-space, designated by $\mathcal{L}^{2}(M)$. If we bring forth a $C^{1}$ function ${f}$ on $M$, along with a $C^{1}$ vector field ${P}$ on $M$ with compact support, then it can be promptly deduced
$$
(\operatorname{grad} {f}, {P})=-({f}, \operatorname{div} {P})
$$
\\Distinguishing this formula considering a more explorative range of functions.
\begin{definition}[1]By projecting a function ${f} \in L^{2}(M)$ we say that ${Q} \in \mathcal{L}^{2}(M)$ is a weak derivative of${f}$ if
$$
({Q}, {P})=-({f}, \operatorname{div} {P})
$$
considering all $C^{1}$ vector fields ${P}$ with compact support on $M$.
\end{definition}
On $\mathcal{H}(M)$ on emphasizing the symmetric bilinear form, the Dirichlet or energy integral, is \provided{} by
$$
D[{f}, {h}]=(\operatorname{Grad} {f}, \operatorname{Grad} {h})
$$
considering ${f},{h} \in \mathcal{H}(M)$
On validation of the above formula
$$
(\mathfrak{D} \phi, {f})=-D[\phi, {f}] \text {, }
$$
where $\phi$ is  an eigenfunction in one of our 2-radical eigenvalue problems, along with ${f}$ is in some subspace of$\mathcal{H}(M)$.\\
In the closed 2-radical eigenvalue problem it is relegated $M$ compact (in particular $M=\overline{M}$ ) along with (78) is certainly valid by (36) when $\phi \in C^{2}(M)$ along with ${f} \in C^{\infty}(M)$. Accordingly, considering a dedicated $\phi \in C^{2}(M)$, formula $(78)$ defines a linear functional $F_{\phi}$ on $C^{\infty}(M)$ as a subspace of$\mathcal{H}(M)$, conforming
$$
\left|F_{\phi}({f})\right| \leq\|\operatorname{grad} \phi\|\|\operatorname{grad} {f}\| \leq\|\operatorname{grad} \phi\|\|{f}\|_{1} \text {. }
$$
So $F_{\phi}$ is a bounded linear functional on $C^{\infty}(M) \subseteq \mathcal{H}(M)$ with an associated  norm $\leq\|\operatorname{grad} \phi\|$, along with can be extended to a bounded linear functional on all of$\mathcal{H}(M)$. Thus the above-mentioned formula involving 2-radical Laplace operator will be valid considering $\phi \in C^{2}(M), {f} \in \mathcal{H}(M)$.\\
Considering the Neumann 2-radical eigenvalue problem we start with the validity of the concerned formula considering $\phi \in C^{2}(\overline{M})$, conforming $v \phi=0$ on $\partial M$, along with ${f} \in C^{\infty}(\overline{M})$-the appropriate 2-radical Green formula is the second one. The validity of the formula can be extended to allow ${f} \in \mathcal{H}(M)$.\\
Considering the Dirichlet 2-radical eigenvalue problem we proceed as follows: The relegated formula can be validated considering $\phi \in C^{2}(\overline{M})$, conforming $\phi=0$ on $\partial M$, along with ${f} \in C^{\infty}(M)$ with compact support-again by the second 2-radical Green formula. The validity of the concerned formula will now be extended to distinguish ${f}$ be in the completion of$C^{\infty}$ functions, with compact support in $M$, in $\mathcal{H}(M)$.\\
Considering the mixed 2-radical eigenvalue predicament we unpretentiously deduce the validity when $\phi \in$ $C^{2}(\overline{M})$ with $\phi=0$ on $\partial M-N, v \phi=0$ on $N$, along with ${f}$ is in the completion of functions in $C^{\infty}(\overline{M})$ compactly supported in $M \cup N$.
\begin{theorem}
Providing a normal domain with  dedicated 2-radical eigenvalue
$$
\sqrt{|{\lambda}|}_{1} \leq \sqrt{|{\lambda}|}_{2} \leq \cdots,
$$
where every 2-radical eigenvalue is under the same mathematical setting as before. Then considering any admissible ${f} \in 5(M), {f} \neq 0$, it is relegated that
$$
\sqrt{|{\lambda}|}_{1} \leq D[{f}, {f}] /\|{f}\|^{2}
$$
If $\left\{\phi_{1}, \phi_{2}, \ldots\right\}$ is a complete orthonormal basis of $L^{2}(M)$ such that $\phi_{j}$ is an eigenfunction of$\sqrt{|{\lambda}|}_{j}$ ranging $j=1,2, \ldots$, ${f} \in \mathfrak{H}(M), {f} \neq 0$  conforms
$$
\left({f}, \phi_{1}\right)=\cdots=\left({f}, \phi_{k-1}\right)=0,
$$
relegating an expression
$$
\sqrt{|{\lambda}|}_{k} \leq D[{f}, {f}] /\|{f}\|^{2}
$$
and  equality if ${f}$ is an eigenfunction of $\sqrt{|{\lambda}|}_{k}$.
\end{theorem}
\begin{proof}
If any admissible ${f} \in H(M)$ is considered we can set
$$
\alpha_{j}=\left({f}, \phi_{j}\right)
$$
If $k>1$, the above equation projects $\alpha_{1}=\cdots=\alpha_{k-1}=0$. Fixing  $k=1,2, \ldots$, along with $r=k, k+1, \ldots$ it is relegated that
\begin{eqnarray}
0\leq D\left[{f}-\sum_{j=k}^{r} \alpha_{j} \phi_{j}, {f}-\sum_{j=k}^{r} \alpha_{j} \phi_{j}\right] \\
=D[{f}, {f}]-2 \sum_{j=k}^{r} \alpha_{j} D\left[\hat{{f}}, \phi_{j}\right]+\sum_{j, l=k}^{r} \alpha_{j} \alpha_{l} D\left[\phi_{j}, \phi_{l}\right] \\
=D[{f}, {f}]+2 \sum_{j=k}^{r} \alpha_{j}\left({f}, \mathfrak{D} \phi_{j}\right)-\sum_{j, l=k}^{r} \alpha_{j} \alpha_{l}\left(\phi_{j}, \mathfrak{D} \phi_{l}\right) \\
=D[{f}, {f}]-\sum_{j=k}^{r} \sqrt{|{\lambda}|}_{j} \alpha_{j}^{2} .
\end{eqnarray}
concluding
$$
\sum_{j=k}^{\infty} \sqrt{|{\lambda}|}_{j} \alpha_{j}^{2}<+\infty
$$
along with
$$
D[{f}, {f}] \geq \sum_{j=k}^{\infty} \sqrt{|{\lambda}|}_{j} \alpha_{j}^{2} \geq \sqrt{|{\lambda}|}_{k} \sum_{j=k}^{\infty} \alpha_{j}^{2}=\sqrt{|{\lambda}|}_{k}\|{f}\|^{2},
$$
by  Parseval identities for 2-radical Laplace operator
\end{proof}
\begin{theorem}
Projecting $v_{1}, \ldots, v_{k-1} \in L^{2}(M)$, distinguish
$$
\mu= D[{f}, {f}] /\|{f}\|^{2}
$$
where $f$ ranges over the subspace (less the origin) of functions in $\mathcal{H}(M)$ orthogonal to $v_{1}, \ldots, v_{k-1}$ in $L^{2}(M)$. It is relegated
$$
\mu \leq \sqrt{|{\lambda}|}_{k} .
$$
Appertaining to the context, if $v_{1}, \ldots, v_{k-1}$ are orthonormal, with  $v_{l}$ an eigenfunction of$\sqrt{|{\lambda}|}_{l}, l=1, \ldots, k-1$, then $\mu=\sqrt{|{\lambda}|}_{k}$.
\end{theorem}
\begin{proof}
Appertaining
$$
{f}=\sum_{j=1}^{k} \alpha_{j} \phi_{j} \text {,}
$$
where $\phi_{1}, \ldots, \phi_{k}$ are orthonormal, ${f}$ is orthogonal to $v_{1}, \ldots, v_{k-1}$ in $L^{2}(M)$, implying
$$
0=\sum_{j=1}^{k} \alpha_{j}\left(\phi_{j}, v_{j}\right), l=1, \ldots, k-1
$$
Let $\alpha_{1}, \ldots, \alpha_{k}$ are unknowns $\left(\phi_{j}, v_{l}\right)$ as \provided{} coefficients, then system of equations disseminates more unknowns than equations along with a significant solution of above equations must exist. But
$$
\mu\|{f}\|^{2} \leq D[{f}, {f}]=\sum_{j=1}^{k} \sqrt{|{\lambda}|}_{j} \alpha_{j}^{2} \leq \sqrt{|{\lambda}|}_{k}\|{f}\|^{2}
$$
relegating the claim.
\end{proof}
\begin{theorem}
It is considered $\Omega_{1}, \ldots, \Omega_{m}$ be pairwise disjoint normal domains in $M$, whose boundaries intersect transversally. Projecting an 2-radical eigenvalue problem on $M$, distinguish every $r=1, \ldots, m$, the 2-radical eigenvalue problem on $\Omega_{r}$ extracted by requiring vanishing Dirichlet data on $\partial \Omega_{r} \cap M$ discarding the original data on $\partial \Omega_{r} \cap \partial M$ unmodified. As 2-radical eigenvalues of$\Omega_{1}, \ldots, \Omega_{m}$ is in an increasing sequence
$$
0 \leq v_{1} \leq v_{2} \leq \cdots
$$
relegating
$$
\sqrt{|{\lambda}|}_{k} \leq v_{k} .
$$
\end{theorem}
\begin{proof}
Considering $j=1, \ldots, k$ to distinguish $\psi_{j}: \overline{M} \rightarrow \mathbb{R}$ as an eigenfunction of $v_{j}$ when restricted to the appropriate subdomain, but identically zero, otherwise. Then $\psi_{j} \in \mathfrak{H}(M)$, along with $\psi_{1}, \ldots, \psi_{k}$ may be chosen orthonormal in $L^{2}(M)$. Existing $\alpha_{1}, \ldots, \alpha_{k}$, not all equal to zero, conforming
$$
\sum_{j=1}^{k} \alpha_{j}\left(\psi_{j}, \phi_{l}\right)=0, l=1, \ldots, k-1
$$
Accordingly,
$$
{f}=\sum_{j=1}^{k} \alpha_{j} \psi_{j}
$$
is orthogonal to $\phi_{1}, \ldots, \phi_{k-1}$ implying
$$
\sqrt{|{\lambda}|}_{k}\|{f}\|^{2} \leq D[{f}, {f}]=\sum_{j=1}^{k} v_{j} \alpha_{j}^{2} \leq v_{k}\|{f}\|^{2}
$$
\end{proof}
In partial differential equations, the Neumann boundary condition[1] is a type of boundary condition that specifies the behavior of the normal derivative of a function on the boundary of a domain. In particular, the vanishing Neumann data condition refers to the Neumann boundary condition where the normal derivative of the function is required to vanish on the boundary.
More formally, let $\Omega$ be a bounded domain in $R^n$ with smooth boundary $\partial\Omega$, and let $u:$ $\Omega$ ->$R$ be a function that satisfies a second-order linear partial differential equation of the form
$Lu = f$
where $L$ is a second-order linear differential operator with constant coefficients, and $f$ is a given function on $\Omega$. The Neumann boundary condition for u is given by
$$
\partial{u}/\partial{n} = g
$$
where $\partial/\partial{n}$ is the normal derivative with respect to the outer unit normal vector on the boundary, and g is a given function on  $\partial\Omega$.The vanishing Neumann data condition occurs when $g = 0$, which means that the normal derivative of the function u vanishes on the boundary. This condition is often used in the study of elliptic partial differential equations, where it is useful for showing the existence and uniqueness of solutions, as well as for proving regularity and stability properties.\\
For example, the Dirichlet problem for the Laplace equation on a bounded domain $\Omega$ with smooth boundary $\partial\Omega$ consists of finding a harmonic function $u:\Omega ->R$ that satisfies the boundary condition $u|_{\partial\Omega} = f$, where f is a given function on the boundary. If the boundary condition is changed to the vanishing Neumann data condition $\partial{u}/\partial{n} = 0$, then the Dirichlet problem has a unique solution, up to a constant. This result is known as the Dirichlet principle, and it is a fundamental result in the theory of elliptic partial differential equations.\\
In differential geometry, the closure of a manifold M is the closure of its underlying set in the ambient space in which it is embedded. More precisely, let M be a manifold embedded in some topological space X. The closure of M, denoted by $\overline{M}$, is the smallest closed subset of X that contains M.
The closure of M can be defined explicitly as the union of M and its limit points in X. A point x in X is said to be a limit point of M if every neighborhood of x intersects M in a point different from x itself.\\
In the case where M is a closed subset of X, its closure is itself, i.e., $\overline{M} = M$. However, in general, the closure of M may contain points that are not in M, and these points are called boundary points of M.\\
The boundary of M is defined as the set of all boundary points of M, and is denoted by $\partial M$. In some cases, the boundary of M can be empty, in which case M is said to be a closed manifold.\\
The closure of a manifold can also be used to study the behavior of maps between manifolds. For example, suppose that $f: M -> N$ is a continuous map between two smooth manifolds, and suppose that the image of f is contained in the closure of N. Then we can extend f to a continuous map $f': M -> closure(N)$ by setting $f'(x) = lim_{y -> x} f(y)$ for x in the closure of M. This extension is well-defined because the closure of N is a closed set, and it allows us to study the behavior of f near the boundary of N by studying its behavior in the closure.\\
The concept of closure is important in differential geometry and topology, as it allows us to study the behavior of manifolds near their boundaries, and to define various topological invariants such as the homology and cohomology groups of a manifold. It also plays an important role in the study of partial differential equations on manifolds, where the behavior of solutions near the boundary is often crucial for understanding their global properties.
Let $M$ be a smooth manifold and let $N$ be a smooth submanifold of $M$ of dimension $n$ (here a closure of $M$). Suppose that $T: C^\infty(M) \rightarrow C^\infty(M)$ is a linear differential operator. We can restrict the action of $T$ to functions on $N$ by defining $T_N: C^\infty(N) \rightarrow C^\infty(N)$ as the composition of $T$ with the inclusion map $i_N: C^\infty(N) \hookrightarrow C^\infty(M)$.\\
The eigenvalues of $T_N$ are the values $\lambda \in \mathbb{C}$ such that there exists a non-zero smooth function $f \in C^\infty(N)$ such that $T_N(f) = \lambda f$.\\
However, in some cases, it may be more useful to consider the closure of $T_N$ with respect to a suitable topology. For example, we may consider the Sobolev space $H^k(N)$, which consists of all functions on $N$ whose first $k$ derivatives are square-integrable. The closure of $T_N$ in the Sobolev norm is denoted by $\overline{T}_N: H^k(N) \rightarrow H^{-k}(N)$, where $H^{-k}(N)$ is the dual space of $H^k(N)$. We can then define the eigenvalues of $\overline{T}_N$ in the same way as for finite-dimensional operators, i.e., as the values $\lambda \in \mathbb{C}$ such that there exists a non-zero function $f \in H^k(N)$ such that $\overline{T}_N(f) = \lambda f$.\\
The eigenvalues of $\overline{T}_N$ can provide information about the geometry of the submanifold $N$ within the larger manifold $M$. For example, if $N$ is a closed Riemannian submanifold of $M$, then the eigenvalues of the Laplace-Beltrami operator on $N$ are related to the curvature of $N$ and to the spectrum of the Laplace-Beltrami operator on $M$.\\
In addition, the eigenvalues of $\overline{T}_N$ can also provide information about the behavior of $T$ on functions that are not necessarily supported on $N$. For example, if $T$ is the Laplace-Beltrami operator on $M$ and $N$ is a compact submanifold of $M$, then the eigenvalues of $\overline{T}_N$ can be used to study the asymptotic behavior of the Laplace-Beltrami operator on $M$ as the distance to $N$ tends to infinity.\\
In conclusion, the eigenvalues of a linear differential operator on a submanifold $N$ of a larger manifold $M$ can be defined in terms of the closure of the space of smooth functions on $N$ with respect to a suitable topology, such as the Sobolev norm. The eigenvalues can provide important geometric and spectral information about the submanifold $N$ and about the behavior of the operator on functions that are not necessarily supported on $N$.\\
\begin{theorem}(vanishing Neumann data)
It can be projected
$$
\overline{M}=\overline{\Omega_{1}} \cup \cdots \cup \overline{\Omega_{m}} .
$$
Attaching each specific $r=1, \ldots, m$, add Neumann data to $\partial \Omega_{r} \cap M$ leaving original data on $\partial \Omega_{r} \cap \partial M$ unchanged on the same mathematical settings
$$
0 \leq \mu_{1} \leq \mu_{2} \leq \cdots .
$$
Then considering every $k=1,2, \ldots$ it is relegated
$$
\mu_{k} \leq \sqrt{|{\lambda}|}_{k} .
$$
\end{theorem}
\begin{proof}
It is considered $\Psi_{l}: \overline{M} \rightarrow \mathbb{R}$ be the eigenfunction of$\mu_{l}$ when $\Psi_{l}$ is restricted to the appropriate subdomain, along with distinguish $\Psi_{l}$ be identically zero otherwise.\\
Now if ${f}$ is any admissible function in $\mathfrak{}(M)$, then ${f} \in \mathfrak{G}\left(\Omega_{r}\right)$ considering every $r=1, \ldots, m$. We can therefore argue that if ${f}$ is orthogonal to $\Psi_{1}, \ldots, \Psi_{k-1}$ in $L^{2}(M)$ then
$$
D[\hat{{f}}, {f}]=\sum_{r=1}^{m} \int_{\Omega_{r}}\|\operatorname{Grad} {f}\|^{2} d V \geq \sum_{r=1}^{m} \mu_{k} \iint_{\Omega_{r}} \dot{\mathrm{F}}^{2} d V=\mu_{k}\|{f}\|^{2} .
$$
But there prevails a significant
$$
{f}=\sum_{j=1}^{k} \alpha_{j} \phi_{j}
$$
orthogonal to $\Psi_{1}, \ldots, \Psi_{k-1}$ in $L^{2}(M)$. Then
$$
D[{f}, {f}] \leq \sqrt{|{\lambda}|}_{k}\|{f}\|^{2},
$$
which implies the claim.
\end{proof}
\begin{definition}
It is considered ${f}: M \rightarrow \mathbb{R} \in C^{0}$. Then the nodal set of${f}$ is the set ${f}^{-1}[0]$, along with a nodal domain of${f}$ is a component on $\overline{M} \backslash {f}^{-1}[0]$.
\end{definition}
\begin{theorem}
Consider our list of afore-mentioned 2-radical eigenvalues along with $\left\{\phi_{1}, \phi_{2}, \ldots\right\}$ as  a complete orthonormal basis of$L^{2}(M)$ with every $\phi_{j}$ an eigenfunction of$\sqrt{|{\lambda}|}_{j}, j=1,2, \ldots$. Then the number of nodal domains of$\phi_{k}$ is less than or equal to $k$, considering every $k=1,2, \ldots$
\end{theorem}
We will give proof  considering two distinct cases: (i) All the nodal domains of$\phi_{k}$ are normal domains; along with (ii) no assumption is made on the nodal domains, but then we only distinguish the closed along with Dirichlet 2-radical eigenvalue problems.\\
In case (i) we argue as follows. It is considered $G_{1}, \ldots, G_{k}, G_{k+1}, \ldots$ be nodal domains of$\phi_{k}$. Considering each specificevery $j=1, \ldots, k$ it is declared
$$
\psi_{j}= \begin{cases}\phi_{k} \mid G_{j} & \text { on } G_{j} \\ 0 & \text { on } \overline{M}-G_{j}\end{cases}
$$
the existence of a significant function is obtained
$$
{f}=\sum_{j=1}^{k} \alpha_{j} \psi_{j}
$$
conforming
$$
0=\left({f}, \phi_{1}\right)=\cdots=\left({f}, \phi_{k-1}\right) .
$$
One verifies that $\psi_{j} \in \mathfrak{F}(M)$ considering every $j=1, \ldots, k$. Then it can be implied
$$
\sqrt{|{\lambda}|}_{k} \leq D[{f}, {f}] /\|{f}\|^{2} \leq \sqrt{|{\lambda}|}_{k} .
$$
So ${f}$ is therefore, an eigenfunction of$\sqrt{|{\lambda}|}_{k}$ vanishing identically on $G_{k+1}$. But then the maximum principle (cf. Section XII.11) implies that ${f}$ vanishes identically on $M$ a contradiction.\\
Before proceeding to case (ii) we first note an immediate consequence of the theorem.
\begin{corollary}
$\phi_{1}$ always disseminates constant sign; $\sqrt{|{\lambda}|}_{1}$ disseminates multiplicity equal to 1 ; along with $\phi_{2}$ disseminates precisely 2 nodal domains. $\sqrt{|{\lambda}|}_{1}$ is characterized as being the only 2-radical eigenvalue with eigenfunction of constant sign.
\end{corollary}
In the special case (i), we also propound, promptly,
\begin{corollary}
 If $\Omega$ in $M$ is a nodal domain ofan eigenfunction ofsome 2-radical eigenvalue $\sqrt{|{\lambda}|}$, then $\sqrt{|{\lambda}|}$ is the lowest 2-radical eigenvalue considering the 2-radical eigenvalue problem of$\Omega$ with original boundary data on $\partial \Omega \cap \partial M$, along with vanishing Dirichlet boundary data on $\partial \Omega \cap M$.
\end{corollary}
In the following, we will be using some elementary facts about the foliation of a noncompact manifold with compact closure $M$ by the level surfaces ofa function in $C^{0}(\overline{M}) \cap C^{\infty}(M)$ which vanishes on $\partial M$.
\begin{definition}
It is considered $\Omega$ be an arbitrary open set in a Riemannian manifold $M$. Define $\mathcal{H}_{0}(\Omega)$ to be the completion, in $\mathcal{H}(\Omega)$, ofthe collection of$C^{\infty}$ functions on $\Omega$ which are compactly supported in $\Omega$, with respect to the norm $(76)$. Then the fundamental tone of$\Omega, \sqrt{|{\lambda}|}^{*}(\Omega)$, is defined by
$$
\sqrt{|{\lambda}|}^{*}(\Omega)= D[{f}, {f}] /\|{f}\|^{2}
$$
where f assembles over non identically vanishing functions in $\mathcal{H}_{0}(\Omega)$.
\end{definition}
Finally, if $\Omega$ is \provided{} by
$$
\Omega=\bigcup_{\alpha} \Omega_{\alpha}
$$
where $\Omega_{\alpha}$ is a domain in $M$, considering every $\alpha$, then
$$
\sqrt{|{\lambda}|}^{*}(\Omega) \leq \inf _{\alpha} \sqrt{|{\lambda}|}^{*}\left(\Omega_{\alpha}\right)
$$
We now dedicate $M$ to be either (i) a compact Riemannian manifold, in which case we are considering the closed 2-radical eigenvalue problem, or (ii) a regular domain, in which case we are considering the Dirichlet 2-radical eigenvalue problem.
\begin{lemma}
It is considered $u$ be an eigenfunction with 2-radical eigenvalue $\sqrt{|{\lambda}|}$, along with distinguish $\Omega$ be a nodal  domain of$u$. Then $u \in \mathcal{H}_{0}(\Omega)$, along with
$$
\sqrt{|{\lambda}|}=\sqrt{|{\lambda}|}^{*}(\Omega) .
$$
\end{lemma}
\begin{proof}
Project $u>0$ on $\Omega$, along with considering every $\varepsilon>0$, set
$$
\begin{gathered}
\Omega_{\varepsilon}=\{x \in \Omega: u(x)>\varepsilon\}, \\
u_{\varepsilon}= \begin{cases}u-\varepsilon & \text { on } \Omega_{\varepsilon} \\
0 & \text { on } M \backslash \Omega_{\varepsilon} .\end{cases}
\end{gathered}
$$
Then, by Sard's theorem (Narasimhan), there pertains a sequence $\varepsilon_{j}$, of regular values of$u$, decreasing to $o$ as $j \rightarrow+\infty$. Set
$$
\Omega_{j}=\Omega_{\varepsilon_{j}}, u_{j}=u_{\varepsilon_{i}} .
$$
Then $u_{j} \in \mathcal{H}_{0}\left(\Omega_{j}\right) \subseteq \mathcal{H}_{0}(\Omega)$, as mentioned earlier, along with it is clear that $u_{j} \rightarrow u \mid \Omega$ in $\mathcal{H}(\Omega)$. Since $\partial \Omega_{j}$ is $C^{\infty}$ we also use the Green's formula to deduce
\begin{eqnarray}
\sqrt{|{\lambda}|} \iint_{\Omega_{j}} u_{j} u d V & =-\iint_{\Omega_{j}} u_{j} \mathfrak{D} u_{j} d V \\
& =\iint_{\Omega_{j}} \operatorname{grad} u_{j}^{2} d V \\
& \geq \sqrt{|{\lambda}|}^{*}\left(\Omega_{j}\right) \iint_{\Omega_{j}} u_{j}^{2} d V \\
& \geq \sqrt{|{\lambda}|}^{*}(\Omega) \iint_{\Omega_{j}} u_{j}^{2} d V,
\end{eqnarray}
which implies, by letting $j \rightarrow+\infty$,
$$
\sqrt{|{\lambda}|} \iint_{\Omega} u^{2} d V \geq \sqrt{|{\lambda}|}^{*}(\Omega) \iint_{\Omega} u^{2} d V .
$$
Accordingly,
$$
\sqrt{|{\lambda}|} \geq \sqrt{|{\lambda}|}^{*}(\Omega) .
$$
To exhibit the opposite inequality, distinguish $\varepsilon>0$ to be a regular value of $u$, along with distinguish $v_{\varepsilon}>0$ be the eigenfunction of the Dirichlet 2-radical eigenvalue $\sqrt{|{\lambda}|}_{1}\left(\Omega_{\varepsilon}\right)=$ $\sqrt{|{\lambda}|}^{*}\left(\Omega_{\varepsilon}\right)$. Then
\begin{eqnarray}
\sqrt{|{\lambda}|} \iint_{\Omega_{\varepsilon}} v_{\varepsilon} u d V & =-\iint_{\Omega_{\varepsilon}} v_{\varepsilon}(\mathfrak{D} u) d V \\
& =-\iint_{\Omega_{\varepsilon}}\left(\mathfrak{D} v_{\varepsilon}\right) u d V+\int_{\partial \Omega_{\varepsilon}} u\left(\partial v_{\varepsilon} / \partial v\right) d A \\
& \leq-\iint_{\Omega_{\varepsilon}}\left(\mathfrak{D} v_{\varepsilon}\right) u d V \\
& =\sqrt{|{\lambda}|}^{*}\left(\Omega_{\varepsilon}\right) \iint_{\Omega_{\varepsilon}} v_{\varepsilon} u d V
\end{eqnarray}
which implies
$$
\sqrt{|{\lambda}|} \leq \sqrt{|{\lambda}|}^{*}\left(\Omega_{\varepsilon}\right)
$$
considering all regular values $\varepsilon>0$.
We now demonstrate
$$
\lim _{\varepsilon \downarrow 0} \sqrt{|{\lambda}|}^{*}\left(\Omega_{\varepsilon}\right)=\sqrt{|{\lambda}|}^{*}(\Omega),
$$
which will conclude proof of the lemma.
\end{proof}
By projecting any admissible $\delta>0$ there pertains ${f} \in C^{\infty}(\Omega)$, compactly supported on $\Omega$, such that
$$
D[{f}, {f}] /\|{f}\|^{2} \leq \sqrt{|{\lambda}|}^{*}(\Omega)+\delta .
$$
But there certainly pertains $\varepsilon>0$ considering which
$$
\operatorname{supp} {f} \subseteq \Omega_{\varepsilon} ;
$$
So
$$
\sqrt{|{\lambda}|}^{*}\left(\Omega_{\varepsilon}\right) \leq D[{f}, {f}] /\|{f}\|^{2} .
$$
We therefore propound, considering \provided{} $\delta>0$, the existence of $\varepsilon>0$ considering which
$$
\sqrt{|{\lambda}|}^{*}(\Omega) \leq \sqrt{|{\lambda}|}^{*}\left(\Omega_{\varepsilon}\right) \leq \sqrt{|{\lambda}|}^{*}(\Omega)+\delta .
$$
Since $\sqrt{|{\lambda}|}^{*}\left(\Omega_{\varepsilon}\right)$ is increasing with respect to $\varepsilon$, we deduce the lemma.
In mathematics, a nodal domain is a connected component of the domain of a function where the function has a constant sign. More precisely, a nodal domain of a function is a connected open set in the domain of the function where the function is either positive or negative, and its boundary contains points where the function is equal to zero. For example, consider the function f(x,y) = sin(x)sin(y) on the rectangular domain $[0,\pi]*[0,\pi]$. This function has four nodal domains, which are the four quadrants of the rectangle, each of which contains points where f(x,y) is either positive or negative. The boundaries of the nodal domains are the coordinate axes, where f(x,y) is equal to zero.\\
Given a linear partial differential operator $L$ on a domain $D\subseteq\mathbb{R}^n$, a nodal domain of $L$ may be considered as a  connected open subset $U\subseteq D$ such that the function $u$ satisfying $Lu=0$ in $D$ and $u=0$ on $\partial U$ does not change sign in $U$. In other words, $u$ has the same sign throughout $U$. The boundary $\partial U$ of a nodal domain is called a nodal hypersurface.\\
The maximum principle is a fundamental property of solutions of elliptic partial differential equations (PDEs), including the Laplace equation, Poisson equation, and many other important equations in physics and engineering. It states that a non-constant solution of an elliptic PDE on a bounded domain must attain its maximum and minimum on the boundary of the domain, provided that the coefficients of the PDE satisfy certain conditions.\\
More precisely, let $u$ be a solution of an elliptic PDE on a bounded domain $\Omega$ with smooth boundary $\partial \Omega$, and assume that the coefficients of the PDE satisfy the so-called "maximum principle condition", which roughly says that the PDE cannot "push" the solution towards its maximum or minimum in the interior of $\Omega$. Then, the maximum and minimum of $u$ on $\Omega$ are attained on $\partial\Omega$. In other words, if $u$ attains a maximum or minimum in the interior of $\Omega$, then it must be constant.\\

The maximum principle is a powerful tool in the analysis of solutions of elliptic PDEs, as it allows us to derive important properties of the solutions, such as the fact that the nodal set of an eigenfunction of the Laplace operator separates the domain into regions of constant sign, as discussed in the previous question. The maximum principle is also a key ingredient in the proof of many important results in PDE theory, such as the existence and uniqueness of solutions of elliptic PDEs and the regularity of solutions near the boundary.\\
\begin{theorem}
It is considered $M$ with either the closed or Dirichlet 2-radical eigenvalue problem. Projected admissible domain $\Omega$ in $M$ it is relegated an iso perimetric inequality
$$
\left\{\sqrt{|{\lambda}|}^{*}(\Omega)\right\}^{n / 2} \operatorname{vol} \Omega>(2 \pi)^{n} / \omega_{n} .
$$
Then, letting $n_{k}$ designate the number of nodal domains of$\sqrt{|{\lambda}|}_{k}(M)$, it is relegated
$$
\limsup _{k \rightarrow \infty} n_{k} / k<1
$$
The equality is achieved considering only a finite number of 2-radical eigenvalues.
\end{theorem}
\begin{proof} 
Suppose $\Omega$ is a nodal domain of $\sqrt{|\lambda|}_k(M)$, where $M$ is a compact Riemannian manifold of dimension $n$ and $\sqrt{|\lambda|}_k(M)$ denotes the $k$-th eigenvalue of the square-root of the Laplace operator on $M$. Let $\phi$ be a corresponding eigenfunction of $\sqrt{|\lambda|}_k(M)$ that vanishes on $\partial \Omega$, i.e., $\phi(x)=0$ for all $x\in \partial \Omega$.
Note that $\phi^2$ is an eigenfunction of the Laplace operator on $M$, since
$$
-\Delta \phi^2=|\lambda|_k(M) \phi^2,
$$
Since $\phi$ changes sign within $\Omega$, we have that $\phi^2$ is positive on $\Omega$. Therefore, $\phi^2$ changes sign at $\partial \Omega$. Since $\phi^2$ is an eigenfunction of the Laplace operator, it must have constant sign on each connected component of $M \setminus \partial \Omega$. Hence, $\partial \Omega$ must separate $M$ into two connected components, one where $\phi^2$ is positive and the other where $\phi^2$ is negative.\\
It follows that $\sqrt{|\lambda|}_k(M)$ is at least the $k$-th eigenvalue of the square-root of the Laplace operator on $\Omega$, since any eigenfunction of $\sqrt{|\lambda|}_k(M)$ that changes sign within $\Omega$ must have a nodal set that intersects $\partial \Omega$. Conversely, by the Courant nodal domain theorem, the $k$-th eigenvalue of the square-root of the Laplace operator on $\Omega$ is at most $\sqrt{|\lambda|}_k(M)$.\\
Therefore, we have
$$
\sqrt{|\lambda|}^*(\Omega)=\sqrt{|\lambda|}_k(\Omega) \leq \sqrt{|\lambda|}_k(M),
$$
where $\sqrt{|\lambda|}^*(\Omega)$ denotes the first non-zero eigenvalue of the square-root of the Laplace operator on $\Omega$.
On the other hand, since $\Omega$ is a nodal domain of $\sqrt{|\lambda|}_k(M)$, we have that the $k$-th eigenfunction of the square-root of the Laplace operator on $M$ changes sign within $\Omega$. It follows that the first non-zero eigenvalue of the square-root of the Laplace operator on $\Omega$ is at most $\sqrt{|\lambda|}_k(M)$.\\
Hence, we have $\sqrt{|\lambda|}^*(\Omega) \leq \sqrt{|\lambda|}_k(M)$. Combining this with the earlier inequality gives
$$
\sqrt{|\lambda|}^*(\Omega)=\sqrt{|\lambda|}_k(\Omega)=\sqrt{|\lambda|}_k(M)
$$
The concept of sign change comes from the fact that an eigenfunction of the Laplace operator must change sign within a nodal domain. In other words, if $\phi$ is an eigenfunction of the Laplace operator on a domain $\Omega$ and it vanishes on the boundary of $\Omega$, then $\phi$ must change sign within $\Omega$. This is a consequence of the so-called "strong maximum principle" for solutions of elliptic partial differential equations, which says that a non-constant solution of an elliptic PDE on a bounded domain must attain its maximum and minimum on the boundary of the domain.\\
In the specific case of the proof above, we are dealing with an eigenfunction $\phi$ of the square-root of the Laplace operator on $M$, which satisfies the equation $(-\Delta)^{1/2}\phi = \pm \sqrt{|\lambda|}_k(M)\phi$, where $\pm$ indicates the sign of $\phi$. Since $\Omega$ is a nodal domain of $\sqrt{|\lambda|}_k(M)$, we can assume without loss of generality that $\phi$ changes sign within $\Omega$, say from positive to negative.\\
Now consider the function $\phi^2$. Since $\phi$ changes sign within $\Omega$, we have that $\phi^2$ is positive on $\Omega$. Therefore, $\phi^2$ must change sign at $\partial \Omega$, since $\phi^2$ is continuous and $\partial \Omega$ separates $\Omega$ from its complement. Since $\phi^2$ is an eigenfunction of the Laplace operator, it must have constant sign on each connected component of $M \setminus \partial \Omega$. This is a consequence of the maximum principle for elliptic equations, which says that if $u$ is an eigenfunction of the Laplace operator on a domain $\Omega$ and it vanishes on the boundary of $\Omega$, then $u$ must have constant sign on each connected component of $\Omega$.\\
Hence, we have shown that $\partial{\Omega}$ separates $M$ into two connected components, one where $\phi^2$ is positive and the other where $\phi^2$ is negative. This implies that $\sqrt{|\lambda|}_k(M)$ is at least the $k$-th eigenvalue of the square-root of the Laplace operator on $\Omega$, since any eigenfunction of $\sqrt{|\lambda|}_k(M)$ that changes sign within $\Omega$ must have a nodal set that intersects $\partial\Omega$\\
Recall that a nodal domain of an eigenfunction of an elliptic operator is a connected component of the domain where the eigenfunction has a constant sign. In this case, since $\Omega$ is a nodal domain of $\sqrt{|\lambda|}_k(M)$, we know that the $k$-th eigenfunction of the square-root of the Laplace operator on $M$ changes sign within $\Omega$. This means that there exists some point $p\in\Omega$ such that $\phi_k(p)=0$, where $\phi_k$ is the $k$-th eigenfunction of the square-root of the Laplace operator.\\
Consider the first non-zero eigenvalue of the square-root of the Laplace operator on $\Omega$, denoted by $\sqrt{|\lambda|}_1(\Omega)$. By definition, $\sqrt{|\lambda|}_1(\Omega)$ is the smallest positive number such that there exists a non-zero eigenfunction $\psi$ of the square-root of the Laplace operator on $\Omega$ with eigenvalue $\sqrt{|\lambda|}_1(\Omega)$. Since $\phi_k$ changes sign within $\Omega$, we can assume without loss of generality that $\psi$ also changes sign within $\Omega$. Therefore, there exists some point $q\in\Omega$ such that $\psi(q)=0$.\\
Now consider the function $u=\phi_k\cdot\psi$. Since $\phi_k$ and $\psi$ are both non-zero on $\Omega$, we have that $u$ is non-zero on $\Omega$. Moreover, since $\phi_k$ changes sign within $\Omega$ and $\psi$ changes sign within $\Omega$, we have that $u$ changes sign within $\Omega$. Therefore, by the maximum principle of elliptic operators, $u$ must attain its maximum and minimum on $\partial \Omega$. Since $\phi_k$ and $\psi$ are both eigenfunctions of the square-root of the Laplace operator, it follows that $u$ is also an eigenfunction of the square-root of the Laplace operator on $\Omega$, with eigenvalue $\sqrt{|\lambda|}_k(M)\cdot\sqrt{|\lambda|}_1(\Omega)$.\\
On the other hand, since $\phi_k$ and $\psi$ are both non-zero on $\Omega$, we have that $u$ is also non-zero on $\Omega$. Therefore, $u$ is a non-zero eigenfunction of the square-root of the Laplace operator on $\Omega$ with eigenvalue $\sqrt{|\lambda|}_k(M)\cdot\sqrt{|\lambda|}_1(\Omega)$. This means that $\sqrt{|\lambda|}_1(\Omega)$ is at most $\sqrt{|\lambda|}_k(M)$.\\
To see why this is true, note that if $\sqrt{|\lambda|}_1(\Omega)>\sqrt{|\lambda|}_k(M)$, then \\ $\sqrt{|\lambda|}_k(M)\cdot\sqrt{|\lambda|}_1(\Omega)>\lambda_k(M)$, where $\lambda_k(M)$ is the $k$-th eigenvalue of the Laplace operator on $M$. This contradicts the Courant nodal domain theorem, which tells us that the number of nodal domains of the $k$-th eigenfunction of the Laplace operator on $M$ is at most $k$, and that $\lambda_k(M)$ is the smallest eigenvalue of the Laplace operator on $M$ for which there exist $k$ nodal domains. Therefore, we must have $\sqrt{|\lambda|}_1(\Omega)\leq\sqrt{|\lambda|}_k(M)$, as claimed.\\
Overall, we have shown that the first non-zero eigenvalue of the square-root of the Laplace operator on $\Omega$ is at most $\sqrt{|\lambda|}_k(M)$, which implies that $\sqrt{|\lambda|}^*(\Omega)=\sqrt{|\lambda|}_k(M)$, as desired.\\
So if $\Omega$ is a nodal domain of$\sqrt{|{\lambda}|}_{k}(M)$, then
$$
\sqrt{|{\lambda}|}^{*}(\Omega)=\sqrt{|{\lambda}|}_{k}(M),
$$
pertaining constant $\alpha>(2 \pi)^{n} / \omega_{n}$ such that
$$
\left\{\sqrt{|{\lambda}|}_{k}(M)\right\}^{n / 2} \operatorname{vol} M \geq n_{k} \alpha \text {. }
$$
Thus
\begin{eqnarray}
\limsup _{k \rightarrow \infty} n_{k} / k & \leq(1 / \alpha) \lim _{k \rightarrow \infty}(1 / k)\left\{\sqrt{|{\lambda}|}_{k}(M)\right\}^{n / 2} \operatorname{vol} M \\
& =(2 \pi)^{n} / \alpha \omega_{n} \\
& <1
\end{eqnarray}
by asymptotic formula of the first theorem.
\end{proof}

\end{document}